\documentclass[12pt,
psamsfonts]{amsart}

\usepackage{amssymb,amsfonts,amsmath}
\usepackage[all,arc]{xy}
\usepackage{enumerate}
\usepackage{mathrsfs}
\usepackage{fullpage}
\usepackage{xspace}

\newtheorem{thm}{Theorem}[section]
\newtheorem{cor}[thm]{Corollary}
\newtheorem{prop}[thm]{Proposition}
\newtheorem{lem}[thm]{Lemma}

\theoremstyle{definition}
\newtheorem{defn}[thm]{Definition}

\theoremstyle{remark}
\newtheorem{rem}[thm]{Remark}

\newcommand{\Z}{\mathbb{Z}\xspace}
\newcommand{\Q}{\mathbb{Q}\xspace}
\DeclareMathOperator{\Spec}{Spec}
\DeclareMathOperator{\res}{res}
\DeclareMathOperator{\Tr}{Tr}
\DeclareMathOperator{\ord}{ord}
\DeclareMathOperator{\Sym}{Sym}
\DeclareMathOperator{\dv}{div}
\DeclareMathOperator{\alb}{alb}
\DeclareMathOperator{\img}{Im}
\DeclareMathOperator{\et}{et}
\DeclareMathOperator{\Cor}{Cor}
\DeclareMathOperator{\Hom}{Hom}

\makeatletter
\let\c@equation\c@thm
\makeatother
\numberwithin{equation}{section}

\newcommand{\textlatin }

\title{ON A FILTRATION OF CH$_{0}$ FOR AN ABELIAN VARIETY}

\author{EVANGELIA GAZAKI}

\begin{document}

\maketitle
\begin{small}\textbf{Abstract:} Let $A$ be an abelian variety defined over a field $k$. In this paper we define a filtration $F^{r}$ of the group $CH_{0}(A)$ and prove an isomorphism $\frac{K(k;A,...,A)}{\Sym}\otimes\mathbb{Z}[\frac{1}{r!}]\simeq F^{r}/F^{r+1}\otimes\mathbb{Z}[\frac{1}{r!}]$, where $K(k;A,...,A)$ is the Somekawa K-group attached to  $r$-copies of the abelian variety $A$.\\ In the special case when $k$ is a finite extension of $\mathbb{Q}_{p}$ and $A$ has split multiplicative reduction, we compute the kernel of the map $CH_{0}(A)\otimes\Z[\frac{1}{2}]\rightarrow \rm{Hom}(Br(A),\Q/\Z)\otimes\Z[\frac{1}{2}]$, induced by the pairing $CH_{0}(A)\times Br(A)\rightarrow\mathbb{Q}/\Z$.
\end{small}
\vspace{12pt}

\section{Introduction}
In \cite{Som}, Somekawa introduced the $K$-group $K(k;G_{1},\dots,G_{n})$ attached to semiabelian varieties $G_{1},\dots, G_{n}$ over a field $k$. In the case when $G_{i}=\mathbb{G}_{m}$ for all $i=1,\dots, n$, there is a canonical isomorphism $K(k;\mathbb{G}_{m},\dots,\mathbb{G}_{m})\simeq K_{n}^{M}(k)$  with the usual Milnor $K$-group. In section \ref{K} we recall the definition in the case when all $G_{i}$ are abelian varieties.

If now $A$ is an abelian variety over a field $k$, we can consider the group $K(k;A,\dots,A)$ attached to $r$ copies of $A$. In this paper, we study the relation between this group and the group $CH_{0}(A)$ of zero cycles modulo rational equivalence on $A$. Both those groups are highly incomputable, so our effort is focusing on obtaining some information for $CH_{0}(A)$ by looking at $K(k;A,\dots,A)$ and vice versa.

More specifically, in section 3, we define a descending filtration $F^{r}$ of the group $CH_{0}(A)$, such that the successive quotients $F^{r}/F^{r+1}$ are "almost" isomorphic to $S_{r}(k;A)$, (Theorem (\ref{iso2})), where by $S_{r}(k;A)$ we denote the quotient of $K(k;A,\dots,A)$ by the subgroup generated by elements of the form $\{x_{1},\dots,x_{r}\}_{k'/k}-\{x_{\sigma(1)},\dots,x_{\sigma(r)}\}_{k'/k}$, where $\sigma:\{1,\dots,r\}\rightarrow\{1,\dots,r\}$ is any permutation. \\
The advantage of our result is that it holds over any base field $k$, allowing us to obtain different corollaries by changing the base field. In the case  of an algebraically closed field $k=\overline{k}$, the filtration $F^{r}$ coincides, after  $\otimes\mathbb{Q}$, with the filtration defined by S. Bloch in \cite{Bl} (proposition \ref{algcl}). Recall that the filtration of Bloch, which we denote by $G^{r}$, was defined (in the case $k=\overline{k}$) as follows:
\begin{eqnarray*}&&G^{0}CH_{0}(A)=CH_{0}(A),\\
&&G^{1}CH_{0}(A)=<[a]-[0]:a\in A>,\\
&&G^{2}CH_{0}(A)=<[a+b]-[a]-[b]+[0]:a,\;b\in A>,\\
&&G^{3}CH_{0}(A)=<[a+b+c]-[a+b]-[a+c]-[b+c]+[a]+[b]+[c]-[0]:a,\;b,\;c\in A>,\\
&&G^{r}CH_{0}(A)=<\sum_{j=0}^{r}(-1)^{r-j}\sum_{1\leq\nu_{1}<\dots<\nu_{j}\leq r}[a_{\nu_{1}}+\dots+a_{\nu_{j}}],
a_{i}\in A>.
\end{eqnarray*}  One other case of particular interest is when $k$ is a finite extension of $\mathbb{Q}_{p}$ and $A$ has split semi-ordinary reduction. In this case, we prove that the successive quotients $F^{r}/F^{r+1}$  are divisible, for $r\geq 3$, while $\displaystyle F^{2}/F^{3}\otimes\mathbb{Z}[\frac{1}{2}]$ is the direct sum of a finite group and a divisible group (corollary (\ref{div2})). This result is deduced by our proposition (\ref{phi2}) and a result of Raskind and Spiess (\cite{Ras/Sp}), where they prove that if $A_{1},\dots,A_{n}$ are abelian varieties over $k$, all satisfying the assumptions stated above, then the group $K(k;A_{1},\dots,A_{n})$ is the direct sum of a  divisible and a finite group, for $n\geq 2$ and it is in fact divisible for $n\geq 3$. \\
We note that the group $F^{2}$ turns out to be the Albanese  kernel of $A$, which in the case of a smooth projective variety over a $p$-adic field $k$ is conjectured to be the direct sum of a finite group and a divisible group.
\subsection{Main Results:}
Our first result gives a nontrivial group homomorphism $CH_{0}(A)\longrightarrow S_{r}(k;A)$, for any $r\geq 0$.
\begin{prop}\label{phi2} Let $k$ be a field and $A$ an abelian variety over $k$. For any $r\geq0$ there is a well defined  abelian group homomorphism \begin{eqnarray*}&&\Phi_{r}:CH_{0}(A)\longrightarrow S_{r}(k;A)\\
&&[a]\longrightarrow\{a,a,\dots,a\}_{k(a)/k},
\end{eqnarray*} where $a$ is any closed point of $A$. For $r=0$, $S_{0}(k,A)=\mathbb{Z}$, and we define $\Phi_{0}$ to be the degree map.
\end{prop} Our next step is to define the filtration $F^{r}$ of $CH_{0}(A)$ by
$F^{r}CH_{0}(A)=\bigcap_{j=0}^{r-1}\ker\Phi_{j}$, for $r\geq 1$. In particular, $F^{0}CH_{0}(A)=CH_{0}(A)$ and $F^{1}CH_{0}(A)$ is the subgroup of degree zero cycles.\\ The next result gives a homomorphism in the reverse direction as follows:
\begin{prop}\label{psi2} Let $r\geq 0$ be an integer. There is a well defined abelian group homomorphism
\begin{eqnarray*}&&\Psi_{r}:S_{r}(k;A)\longrightarrow \frac{F^{r}CH_{0}(A)}{F^{r+1}CH_{0}(A)}\\
&&\{a_{1},\dots,a_{r}\}_{k'/k}\longrightarrow\sum_{j=0}^{r}(-1)^{r-j}\Tr_{k'/k}(\sum_{1\leq\nu_{1}<\dots<\nu_{j}\leq r}[a_{\nu_{1}}+\dots+a_{\nu_{j}}]_{k'}),
\end{eqnarray*} where the summand corresponding to $j=0$  is $(-1)^{r}\Tr_{k'/k}([0]_{k'})$. Moreover, the homomorphisms $\Psi_{r}$ satisfy the property, $\Phi_{r}\circ\Psi_{r}=\cdot r!$ on $S_{r}(k;A)$.
\end{prop}  In particular we see that after applying the functor $\displaystyle\otimes\mathbb{Z}[\frac{1}{r!}]$, the morphisms $\Phi_{r}$ and $\Psi_{r}$ induce the following isomorphisms.
\begin{thm}\label{iso2} Let $k$ be a field and $A$ an abelian variety over $k$. For the filtration $F^{r}CH_{0}(A)$ defined above, there are canonical isomorphisms of abelian groups:
$$\Phi_{r}:\mathbb{Z}[\frac{1}{r!}]\otimes\frac{F^{r}}{F^{r+1}}\stackrel{\simeq}{\longrightarrow}\mathbb{Z}[\frac{1}{r!}]
\otimes S_{r}(k;A), \;r\geq 0,$$ with $\displaystyle\Phi_{r}^{-1}=\frac{1}{r!}\Psi_{r}$. Moreover, the group $F^{2}CH_{0}(A)$ is precisely the Albanese kernel of $A$.
\end{thm}  As we will see in section 3, Theorem \ref{iso2} can be easily deduced by the two previous propositions.
\subsection{Corollaries} In sections 4 and 5 we obtain various corollaries and properties of the filtration $F^{r}$. After treating the case of an algebraically closed field $k$, in the last part of section \ref{4}, we describe a recursive algorithm to compute generators of the group $\displaystyle F^{r}\otimes\mathbb{Z}[\frac{1}{(r-1)!}]$, for $r\geq 1$.  \\
In section 5, we use the Somekawa map $s_{n}$ (see section 6 for a definition), to obtain a cycle map to Galois cohomology $$\frac{F^{r}/F^{r+1}}{n}\longrightarrow H^{r}(k,\bigwedge^{r}A[n]),$$ where $n$ is any integer invertible in $k$.
\subsection{The p-adic Case}
In our last section, we obtain some results when the base field $k$ is a finite extension of $\mathbb{Q}_{p}$. Using a result of W. Raskind and M. Spiess, \cite{Ras/Sp}, we obtain the following corollary.
\begin{cor}\label{div2} Let $A$  be an abelian variety over a $p$-adic field $k$ having split semi-ordinary reduction. Then for the filtration defined above, it holds:
\begin{enumerate}\item For $r\geq 3$, the groups $F^{r}/F^{r+1}$ are divisible.
\item The group $\displaystyle F^{2}/F^{3}\otimes\mathbb{Z}[\frac{1}{2}]$ is the direct sum of a divisible group and a finite group.
\end{enumerate}
\end{cor}
Using these divisibility results, we move on to compute the kernel of the map $$CH_{0}(A)\otimes\Z[\frac{1}{2}]\rightarrow \rm{Hom}(Br(A),\mathbb{Q}/\Z)\otimes\Z[\frac{1}{2}],$$ induced by the Brauer-Manin pairing
$<,>_{A}:CH_{0}(A)\times Br(A)\rightarrow\mathbb{Q}/\Z$, in the special case when $A$ has split multiplicative reduction. (For a definition of the pairing see section 6.2). We obtain the following theorem.
\begin{thm} Let $A$ be an abelian variety over $k$. The subgroup $F^{3}$ is contained in the kernel of the map
$$j:CH_{0}(A)\rightarrow \rm{Hom}(Br(A),\mathbb{Q}/\Z).$$ If moreover $A$ has split multiplicative reduction, then the kernel of the map $$CH_{0}(A)\otimes\Z[\frac{1}{2}]\stackrel{j\otimes\mathbb{Z}[\frac{1}{2}]}{\longrightarrow} \rm{Hom}(Br(A),\mathbb{Q}/\Z)\otimes\Z[\frac{1}{2}]$$ is the subgroup $D$ of $\displaystyle F^{2}\otimes\Z[\frac{1}{2}]$, which contains $\displaystyle F^{3}\otimes\Z[\frac{1}{2}]$ and is such that $\displaystyle D/(F^{3}\otimes\Z[\frac{1}{2}])$ is the maximal divisible subgroup of $\displaystyle F^{2}/F^{3}\otimes\Z[\frac{1}{2}]$.

\end{thm}
We point out that our result was motivated by a result of  W. Raskind and M. Spiess, who in \cite{Ras/Sp} obtain an isomorphism
$$\mathbb{Z}\oplus\bigoplus_{1\leq\nu\leq d}\bigoplus_{1\leq i_{1}<\dots<i_{\nu}\leq d}K(k;J_{i_{1}},\dots,J_{i_{\nu}})
\simeq CH_{0}(C_{1}\times\dots\times C_{d}),$$ where  $C_{1},\dots,C_{d}$ are smooth, projective, geometrically connected curves over $k$, all having a $k$-rational point, and $J_{j}$ is the Jacobian variety of $C_{j}$. This isomorphism induces a descending filtration of the group $CH_{0}(C_{1}\times\dots\times C_{d})$ in terms of the $K$-groups $K(k;J_{i_{1}},\dots,J_{i_{\nu}})$.\\
Moreover, our results concerning the pairing of $CH_{0}(A)$ with  $Br(A)$ were motivated by a result of T. Yamazaki, who in \cite{Yam} computes the kernel of the map $j:CH_{0}(X)\rightarrow Br(X)^{\star}$, when $X=C_{1}\times\dots\times C_{d}$ is a product of Mumford curves.
\vspace{4pt}
\subsection{Notation} If $k$ is a field, we denote by $\overline{k}$ its algebraic closure. If $L\supset k$ is any field extension and  $X$ is a variety over $k$, we denote by $X_{L}=X\times_{k}\Spec L$ its base change to $L$. Moreover, if $a$ is any closed point of $X$, we denote by $k(a)$ its residue field.\\
If $B$ is a discrete abelian group, we denote by $B^{\star}$ the group $\rm{Hom}(B,\mathbb{Q}/\Z)$. If $X$ is a smooth variety, and $\mathcal{F}$ is an abelian sheaf on the \'{e}tale site of $X$, we denote $H^{r}(X,\mathcal{F}),\;r\geq 0$, the \'{e}tale cohomology groups of $X$ with coefficients in $\mathcal{F}$.\\
If now $A$ is an abelian variety over $k$, and $n$ is any integer, we denote by $A[n]=\ker(A\stackrel{n}{\longrightarrow}A)$ the $n$-torsion points of $A$.  Further, we denote by $\widehat{A}$, the dual abelian variety of $A$.

\vspace{13pt}
\section{Review of Definitions}In this section we recall the definition of the Albanese variety, $Alb_{X}$, of a smooth projective variety $X$ over a field $k$, as well as the definition of the Somekawa K-group $K(k;A_{1},...,A_{n})$ attached to abelian varieties.
\subsection{The Albanese map:}
Let $X$ be a smooth projective variety of dimension $d\geq 0$ over a field $k$. We consider the group $CH_{0}(X)$ of zero cycles modulo rational equivalence on $X$. We denote $A_{0}(X)$ the subgroup of $CH_{0}(X)$ of degree zero cycles. For a concise collection of results concerning the groups $CH_{0}(X)$ and $A_{0}(X)$, we refer to \cite{CT1}.\\
We recall that if $X$ is a smooth projective variety having a $k$ rational point $P$, there is a unique abelian variety $Alb_{X}$, called the Albanese variety of $X$, and a unique morphism $\varphi:X\rightarrow Alb_{X}$ taking $P$ to the zero element of $Alb_{X}$, satisfying the universal property: If $f:X\rightarrow B$ is a morphism of $X$ to an abelian variety $B$ taking  $P$ to the zero element of $B$,  then $f$ factors uniquely through $\varphi$. We note that the map $\varphi$ induces a group homomorphism $\alb_{X}:A_{0}(X)\rightarrow Alb_{X}(k)$, not depending on the $k$-rational point $P$, called the albanese map of $X$. For a proof of this statement and more details on $Alb_{X}$ we refer to \cite{Bl2}.\\ In the case of a smooth projective curve $C$ having a $k$-rational point, we know that $alb_{C}$ gives an isomorphism of $Pic^{0}(C)$ with the usual Jacobian of $C$. In higher dimensions, this map is far from being injective. Finally we notice, that if $A$ is an abelian variety, then $A$ is its own Albanese variety.
\subsection{The group $K(k;A_{1},...,A_{n})$:}
\label{K} Here we review the definition of the $K$-group $K(k;A_{1},...,A_{n})$ attached to abelian varieties $A_{1},\dots,A_{n}$, which was first introduced by  Somekawa in \cite{Som}. Notice that if $k'/k$ is any extension of $k$, there is a natural restriction morphism $\res_{k'/k}:A(k)\rightarrow A(k')$, while if $k'\supset k$ is a finite extension, we obtain further a well defined Trace morphism $\Tr_{k'/k}:A(k')\rightarrow A(k)$. We define
$$K(k;A_{1},...,A_{n})=[\bigoplus_{k'/k}A_{1}(k')\otimes...\otimes A_{n}(k')]/R,$$ where the sum extends over all finite extensions $k'\supset k$  and $R$ is the subgroup generated by the following two families of elements:
\begin{enumerate}\item\label{K-1} If $L\supset E\supset k$ are two finite extensions of $k$ and we have points $a_{i}\in A_{i}(L)$, for some $i\in\{1,2,...,n\}$, and $a_{j}\in A_{j}(E)$, for all $j\neq i$, then
$$a_{1}\otimes...\otimes \Tr_{L/E}(a_{i})\otimes...\otimes a_{n}-\res_{L/E}(a_{1})\otimes...\otimes a_{i}\otimes...\otimes \res_{L/E}(a_{n})\in R$$
\item\label{K-2} Let $K/k$ be a function field in one variable over $k$. Let $f\in K^{\times}$ and $x_{i}\in A_{i}(K),\;i=1,...,n$. Then we define
    $$\sum_{v\;\rm{place\;of\;K/k}}\ord_{v}(f)(s_{v}^{1}(x_{1})\otimes...\otimes s_{v}^{n}(x_{n}))\in R,$$ where the sum extends over all places $v$ of $K$ over $k$. If $v$ is such a place, the morphisms $s_{v}^{i}$ are specialization maps $A_{i}(K)\rightarrow A_{i}(k_{v})$, where $k_{v}$ is the residue field of the place $v$, and are defined as follows: Let $K_{v}$ be the completion of $K$ with respect to the valuation $v$ and $\mathcal{O}_{v}$ be its ring of integers. The properness of $A_{i}$  over $k$ yields isomorphisms $A_{i}(K_{v})\simeq A_{i}(\mathcal{O}_{v}),$ for all $i=1,...,n$. Further, we have a natural map $A_{i}(\mathcal{O}_{v})\rightarrow A_{i}(k_{v})$ induced by $\mathcal{O}_{v}\twoheadrightarrow k_{v}$. This gives,
\[ \xymatrix{
& A_{i}(K)\ar[dr]_{s_{v}^{i}} \ar[r]   &   A_{i}(\mathcal{O}_{v})\ar[d]\\
&  & A_{i}(k_{v}), \\
}
\] where the horizontal map is the composition $A_{i}(K)\stackrel{\res}{\longrightarrow}A_{i}(K_{v})
\stackrel{\simeq}{\longrightarrow}A_{i}(\mathcal{O}_{v})$.
\end{enumerate}
\medskip
\textbf{Notation:}\begin{enumerate}\item The elements of $K(k;A_{1},...,A_{n})$ will be from now on denoted as symbols $\{a_{1},...,a_{n}\}_{k'/k}$, for $a_{i}\in A_{i}(k')$.
\item If $A_{1}=A_{2}=\dots=A_{r}$, then we introduce the notation
\begin{eqnarray*}K_{r}(k;A)=K(k;\overbrace{A,\dots,A}^{r}).
\end{eqnarray*}
Furthermore, we consider the group $$S_{r}(k;A)=\frac{K_{r}(k;A)}{<\{x_{1},\dots,x_{r}\}_{k'/k}-\{x_{\sigma(1)},\dots,x_{\sigma(r)}\}_{k'/k}:\sigma\in\sum_{r}>},$$ where $\Sigma_{r}$ is the group of permutations of the set $\{1,\dots,r\}$.
\item If $a$ is a closed point of $A$, we will denote by $[a]$ the class of $a$ in $CH_{0}(A)$.
\end{enumerate}
\medskip
\textbf{Functoriality:} Let $L/k$ be a  finite extension of $k$. Then there is a well defined trace map defined as follows:
\begin{eqnarray*}\Tr_{L/k}:&&K(L;A_{1},...,A_{n})\rightarrow K(k;A_{1},...,A_{n})\\
&&\{a_{1},\dots,a_{n}\}_{E/L}\rightarrow\{a_{1},\dots,a_{n}\}_{E/k}.
\end{eqnarray*} \\
Moreover, if $j:k\hookrightarrow L$ is any field extension, then there is a well defined Restriction map
$$\res_{L/k}:K(k;A_{1},...,A_{n})\rightarrow K(L;A_{1}\times_{k} L,...,A_{n}\times_{k} L).$$ To define it, consider a finite extension $k'/k$  of $k$ and let $\{a_{1},\dots,a_{n}\}_{k'/k}\in K(k;A_{1},...,A_{n})$. We can write $k'\otimes L=\prod_{i=1}^{m}B_{i}$, where $B_{i}$ are Artin local rings over $L$ of length $e_{i}$, for $i=1,\dots,m$. The residue field $L_{i}$ of $B_{i}$ is a finite extension of $L$, for $i=1,\dots,m$ and an extension of $k'$. We define:
$$\res_{L/k}(\{a_{1},\dots,a_{n}\}_{k'/k})=\sum_{i=1}^{m}e_{i}\{\res_{L_{i}/k'}(a_{1}),\dots,
\res_{L_{i}/k'}(a_{n})\}_{L_{i}/L}.$$ Notice that if $L/k$ is a finite extension of $k$ and $\{a_{1},\dots,a_{n}\}_{k'/k}\in K(k;A_{1},...,A_{n})$, then we have:
$$\Tr_{L/k}(\res_{L/k}(\{a_{1},\dots,a_{n}\}_{k'/k}))=(\sum_{i=1}^{m}e_{i})\{a_{1},\dots,a_{n}\}_{k'/k}=[L:k]
\{a_{1},\dots,a_{n}\}_{k'/k}.$$
\\
\textbf{The Somekawa Map:}
If $n>0$ is an integer which is invertible in $k$, then there is a well defined Galois symbol map:
$$s_{n}:\frac{K(k;A_{1},...,A_{r})}{n}\longrightarrow H^{r}(k, A_{1}[n]\otimes\dots\otimes A_{r}[n]),$$ defined using the cup product of Galois cohomology and the Kummer sequence for abelian varieties. A precise definition of the map $s_{n}$ will be reviewed in section 6.
\begin{rem} In \cite{Som}, Somekawa states the conjecture that the the map $s_{n}$ is injective. This has been proved in some cases. We refer to \cite{Mur/Ram} and \cite{Yam} for some examples where the conjecture holds. On the contrary, in
\cite{Sp/Yam}, M.Spiess and T.Yamazaki provided a counterexample, by constructing a torus $T$ over a field $k$ that has the property that the Galois symbol map $K(k;T,T)/n\stackrel{s_{n}}{\longrightarrow} H^{2}(k,T[n]^{\otimes 2})$ fails to be injective.
\end{rem}
\textbf{Convention-Notation:} Let $k$ be any field and $A$  a variety over $k$. If $a$ is a closed point of $A$, then $a$ induces a unique $k(a)$-rational point $\widetilde{a}$ of $A_{k(a)}$, and for the push-forward map $\Tr_{k(a)/k}:CH_{0}(A_{k(a)})\rightarrow CH_{0}(A)$, it holds $\Tr_{k(a)/k}([\widetilde{a}])=[a]$. \\
If now $k'\supset k(a)\supset k$ is a finite extension, then $a$ can be considered by restriction as a $k'$-rational point of $A$. We will denote by $[a]_{k'}$ the class of $[\res_{k'/k(a)}(a)]$ in $CH_{0}(A_{k'})$. Notice that for the push-forward map, $\Tr_{k'/k}:CH_{0}(A_{k'})\rightarrow CH_{0}(A)$,
it holds $\Tr_{k'/k}([a]_{k'})=[k':k(a)]\cdot[a]$. (See \cite{Ful}, section 1.4). The necessity of this remark will become apparent in proposition \ref{Gr}.
\vspace{6pt}
\section{The canonical Isomorphisms}\label{2}
In this section we define a filtration $F^{r}CH_{0}(A)$ of $CH_{0}(A)$ and prove the existence of canonical morphisms $\Phi_{r}:F^{r}/F^{r+1}\rightarrow S_{r}(k;A)$,  and $\Psi_{r}:S_{r}(k;A)\rightarrow F^{r}/F^{r+1}$, for all $r\geq 0$, so that $\Phi_{r}$ and $\Psi_{r}$ become "almost" each other inverses. So as not to exclude $r=0$ from what it follows, we define $S_{0}(k;A)=\mathbb{Z}$.

\begin{prop}\label{phi} Let $k$ be a field and $A$ an abelian variety over $k$. For any $r\geq0$ there is a well defined  abelian group homomorphism \begin{eqnarray*}&&\Phi_{r}:CH_{0}(A)\longrightarrow S_{r}(k;A)\\
&&[a]\longrightarrow\{a,a,\dots,a\}_{k(a)/k}.
\end{eqnarray*}
\end{prop}
\begin{proof} For $r=0$ we define $\Phi_{0}:CH_{0}(A)\longrightarrow \mathbb{Z}$ to be the degree map.
Let now $r>0$ be a fixed integer. We define a map $Z_{0}(A)\stackrel{\phi_{r}}{\longrightarrow}S_{r}(k;A)$ first at the level of cycles as follows. Let $a$ be any closed point of $A$ with residue field $k(a)$. Then we define $\phi_{r}(a)=\{a,a,\dots,a\}_{k(a)/k}$. To check that $\phi_{r}$ factors through rational equivalence, let $C\subset A$ be a closed irreducible curve with function field $K=k(C)$ and let $f\in K^{\times}$. Let $\widetilde{C}$ be the normalization of $C$ and let $p$
\[ \xymatrix{
&\widetilde{C}\ar[d] \ar[r]^{p} & A\\
& C\ar@{^{(}->}[ru] \\
}
\] be the canonical map. We need to show $\phi_{r}(p_{\star}(\dv(f))=0$. By the definition of $\phi_{r}$ we obtain: \begin{eqnarray*}&&\phi_{r}(p_{\star}(\dv(f))=\phi_{r}(\sum_{x\in\widetilde{C}}\ord_{x}(f)[k(x):k(p(x))][p(x)])=\\
&&\sum_{x\in\widetilde{C}}\ord_{x}(f)[k(x):k(p(x))]\{p(x),\dots,p(x)\}_{k(p(x))/k}=\\
&&\sum_{x\in\widetilde{C}}\ord_{x}(f)\{[k(x):k(p(x))]p(x), p(x),\dots,p(x)\}_{k(p(x))/k}=\\
&&\sum_{x\in\widetilde{C}}\ord_{x}(f)\{\Tr_{k(x)/k(p(x))}(\res_{k(x)/k(p(x))}(p(x)),p(x),\dots,p(x)\}_{k(p(x))/k}=\\
&&\sum_{x\in\widetilde{C}}\ord_{x}(f)\{\res_{k(x)/k(p(x))}(p(x)),\dots,\res_{k(x)/k(p(x))}(p(x))\}_{k(x)/k}.
\end{eqnarray*} Let $\Spec K\stackrel{\eta}{\hookrightarrow}\widetilde{C}$ be the generic point inclusion and let $x$ be a closed point of $\widetilde{C}$. Let $K_{x}$ be the completion of $K$ at the place $x$ and $\mathcal{O}_{K_{x}}$ its ring of integers. Then  the diagram
\[ \xymatrix{
&\Spec K_{x}\ar[d]^{\eta_{x}} \ar[r]  & \Spec K\ar@{^{(}->}[ld]\\
&\widetilde{C} \\
}
\]
yields a $K_{x}$-rational point $p\eta_{x}$ of $A$. The valuative criterion for properness gives a unique $\mathcal{O}_{K_{x}}$-valued point of $A$,
\[ \xymatrix{
&\Spec K_{x}\ar[d] \ar[r]^{p\eta_{x}}  & A\ar[d]\\
&\Spec \mathcal{O}_{K_{x}}\ar[r]\ar[ru]^{\exists !p_{x}} & \Spec k. \\
}
\] Then, we claim that for  the specialization map $s_{x}$ corresponding to the valuation $x$, it holds $s_{x}(p\eta)=\res_{k(x)/k(p(x))}(p(x))$. To see this, we follow the composition
\begin{eqnarray*} &&A(K)\stackrel{\res}{\longrightarrow} A(K_{x})\stackrel{\simeq}{\longrightarrow} A(\mathcal{O}_{K_{x}})
\stackrel{\res}{\longrightarrow} A(k_{x})\\
&&p\eta\;\;\;\;\longrightarrow p\eta_{x}\;\;\;\;\longrightarrow\;\;\;\;p_{x}\longrightarrow\;\;\;\;
\res_{k(x)/k(p(x))}(p(x)).
\end{eqnarray*} This in turn yields:
\begin{eqnarray*}&&\phi_{r}(p_{\star}(\dv(f))=\sum_{x\in\widetilde{C}}\ord_{x}(f)\{\res_{k(x)/k(p(x))}(p(x)),\dots,
\res_{k(x)/k(p(x))}(p(x))\}_{k(x)/k}=\\
&&\sum_{x\in\widetilde{C}}\ord_{x}(f)\{s_{x}(p\eta),\dots,s_{x}(p\eta)\}_{k(x)/k}=0,
\end{eqnarray*} where the last equality comes from the defining relation (\ref{K-2}) of the K-group $K_{r}(k;A)$. We thus obtain a homomorphism $CH_{0}(A)\stackrel{\Phi_{r}}{\longrightarrow}S_{r}(k;A)$ as desired.
\end{proof}
\medskip
\begin{defn}\label{fil} We define a descending filtration $F^{r}$ of $CH_{0}(A)$ by $F^{r}=\bigcap_{j=0}^{r-1}\ker\Phi_{j},\;r\geq 0$. In particular, $F^{0}CH_{0}(A)=CH_{0}(A)$ and $F^{1}CH_{0}(A)=A_{0}(A)$ is the subgroup of degree zero elements.
\end{defn}
\medskip
\begin{prop}\label{Gr} The filtration $F^{r}CH_{0}(A)$ just defined contains the filtration $G^{r}CH_{0}(A)$ defined as follows:
\begin{eqnarray*}&&G^{0}CH_{0}(A)=CH_{0}(A),\\
&&G^{1}CH_{0}(A)=<\Tr_{k'/k}([a]_{k'}-[0]_{k'}):a\in A(k')>,\\
&&G^{2}CH_{0}(A)=<\Tr_{k'/k}([a+b]_{k'}-[a]_{k'}-[b]_{k'}+[0]_{k'}):a,\;b\in A(k')>,\\
&&...\\
&&G^{r}CH_{0}(A)=<\sum_{j=0}^{r}(-1)^{r-j}\Tr_{k'/k}(\sum_{1\leq\nu_{1}<\dots<\nu_{j}\leq r}[a_{\nu_{1}}+\dots+a_{\nu_{j}}]_{k'}):a_{1},\dots,a_{r}\in A(k')>,
\end{eqnarray*} where the summand corresponding to $j=0$  is $(-1)^{r}\Tr_{k'/k}([0]_{k'})$, and $k'$ runs through all finite extensions of $k$.
\end{prop}
\begin{proof} The claim is clear for $r=0$. Let $r\geq 1$ and let $a_{1},\dots,a_{r}\in A(k')$. We denote by $\Phi_{r-1}^{k'}$ the map $CH_{0}(A_{k'})\rightarrow S_{r-1}(k';A\times_{k}k')$ defined as in proposition \ref{phi}. We claim that:
\begin{eqnarray*}\Phi_{r-1}(\sum_{j=0}^{r}(-1)^{r-j}\Tr_{k'/k}(\sum_{1\leq\nu_{1}<\dots<\nu_{j}\leq r}[a_{\nu_{1}}+\dots+a_{\nu_{j}}]_{k'}))=\\
\sum_{j=0}^{r}(-1)^{r-j}\Tr_{k'/k}(\Phi_{r-1}^{k'}
(\sum_{1\leq\nu_{1}<\dots<\nu_{j}\leq r}[a_{\nu_{1}}+\dots+a_{\nu_{j}}]_{k'}))=0.
\end{eqnarray*} The last equality is deduced by the  multilinearity of the symbol $\{x_{1},\dots,x_{r-1}\}_{k'/k'}$ and the fact that $\Phi_{r-1}$ is a group homomorphism.
\\
To justify the first equality, we need to verify that $\Tr_{k'/k}(\Phi_{r-1}^{k'}([a]_{k'}))=\Phi_{r-1}(\Tr_{k'/k}([a]_{k'}))$, where $a\in A(k')$ is a $k'$-rational point of $A$.  Notice that in general the residue field $k(a)$ might be strictly smaller than $k'$. (See Convention-Notation at the end of section 2). We have:
\begin{eqnarray*} \Tr_{k'/k}(\Phi_{r}^{k'}([a]_{k'}))=&&\Tr_{k'/k}(\{\res_{k'/k(a)}(a),\dots,\res_{k'/k(a)}(a)\}_{k'/k'})=\\&&
\{\res_{k'/k(a)}(a),\dots,\res_{k'/k(a)}(a)\}_{k'/k}.\\
\Phi_{r}(\Tr_{k'/k}([a]_{k'}))=&&\Phi_{r}([k':k(a)]\cdot[a]))=[k':k(a)]\{a,\dots,a\}_{k(a)/k}=\\
&&\{[k':k(a)]a,\dots,a\}_{k(a)/k}=\{\Tr_{k'/k(a)}(\res_{k'/k(a)}(a)),a,\dots,a)\}_{k(a)/k}=\\
&&\{\res_{k'/k(a)}(a),\dots,\res_{k'/k(a)}(a)\}_{k'/k}.
\end{eqnarray*}

\end{proof}
\medskip
\begin{prop}\label{psi} Let $r\geq 0$ be an integer. There is a well defined abelian group homomorphism
\begin{eqnarray*}&&\Psi_{r}:S_{r}(k;A)\longrightarrow \frac{F^{r}CH_{0}(A)}{F^{r+1}CH_{0}(A)}\\
&&\{a_{1},\dots,a_{r}\}_{k'/k}\longrightarrow\sum_{j=0}^{r}(-1)^{r-j}\Tr_{k'/k}(\sum_{1\leq\nu_{1}<\dots<\nu_{j}\leq r}[a_{\nu_{1}}+\dots+a_{\nu_{j}}]_{k'}),
\end{eqnarray*} where the summand corresponding to $j=0$  is $(-1)^{r}\Tr_{k'/k}([0]_{k'})$. Moreover, the homomorphisms $\Psi_{r}$ satisfy the property, $\Phi_{r}\circ\Psi_{r}=\cdot r!$ on $S_{r}(k;A)$.
\end{prop}
\begin{proof} \underline{Step 1:}
We define a map
\begin{eqnarray*}&&\Psi_{r}:\bigoplus _{k'/k}(A(k')\times A(k')\times\dots\times A(k'))\longrightarrow \frac{F^{r}CH_{0}(A)}{F^{r+1}CH_{0}(A)}\\
&&(a_{1},\dots,a_{r})\longrightarrow\sum_{j=0}^{r}(-1)^{r-j}\Tr_{k'/k}(\sum_{1\leq\nu_{1}<\dots<\nu_{j}\leq r}[a_{\nu_{1}}+\dots+a_{\nu_{j}}]_{k'}),
\end{eqnarray*} where the direct sum extends over all finite extensions of $k$.  Notice that the inclusion $G^{r+1}\subset F^{r+1}$ proved in proposition \ref{Gr}, forces the  map $\Psi_{r}$ to be multilinear, and thus we obtain a well defined map $$\Psi_{r}:
\bigoplus _{k'/k}(A(k')\otimes A(k')\otimes\dots\otimes A(k'))\longrightarrow \frac{F^{r}CH_{0}(A)}{F^{r+1}CH_{0}(A)}.$$
\underline{Step 2:} We claim that the composition $$\bigoplus _{k'/k}(A(k')\otimes A(k')\otimes\dots\otimes A(k'))\stackrel{\Psi_{r}}{\longrightarrow} \frac{F^{r}CH_{0}(A)}{F^{r+1}CH_{0}(A)}\stackrel{\Phi_{r}}{\longrightarrow}S_{r}(k;A)$$
sends $a_{1}\otimes\dots\otimes a_{r}$ to $r!\{a_{1},\dots,a_{r}\}_{k'/k}$. For, we observe that,
\begin{eqnarray*}&&\Phi_{r}([\sum_{i=1}^{r}a_{i}])=\{\sum_{i=1}^{r}a_{i},\dots,\sum_{i=1}^{r}a_{i}\}_{k'/k}=
\sum_{i_{1}=1}^{r}\sum_{i_{2}=1}^{r}\dots\sum_{i_{r}=1}^{r}
\{a_{i_{1}},\dots,a_{i_{r}}\}_{k'/k}
\end{eqnarray*} and by a combinatorial counting we can see that the only terms of this sum that do not get canceled by $\Phi_{r}(\sum_{j=0}^{r-1}(-1)^{r-j}(\sum_{1\leq\nu_{1}<\dots<\nu_{j}\leq r}[a_{\nu_{1}}+\dots+a_{\nu_{j}}]_{k'}))$ are those where all the $a_{i_{l}}$ are distinct. Thus, using the symmetry of the symbol in $S_{r}(k;A)$, we get all the possible combinations of the set $\{a_{1},\dots,a_{r}\}$ without repetition, which are exactly $r!$.\\
 \\
Notice that the above property forces the elements of the form $(a_{1}\otimes\dots\otimes\Tr_{E/L}(a_{i})\otimes\dots\otimes a_{r})$ and $\res_{E/L}(a_{1})\otimes\dots\otimes a_{i}\otimes\dots\otimes \res_{E/L}(a_{r})$ to have the same image under $\Psi_{r}$, where $E\supset L\supset k$ is a tower of finite extensions,
$a_{i}\in A(E)$ and $a_{j}\in A(L)$, for all $j\neq i$. For,
\begin{eqnarray*}\Phi_{r}\circ \Psi_{r}((a_{1}\otimes\dots\otimes\Tr_{E/L}(a_{i})\otimes\dots\otimes a_{r}))&&=r!\{a_{1},\dots,\Tr_{E/L}(a_{i}),\dots, a_{r}\}_{L/k}\\
&&=r!\{\res_{E/L}(a_{1}),\dots, a_{i},\dots, \res_{E/L}(a_{r})\}_{E/k}\\&&=\Phi_{r}\circ \Psi_{r}(\res_{E/L}(a_{1})\otimes\dots\otimes a_{i}\otimes\dots\otimes \res_{E/L}(a_{r})).
\end{eqnarray*}
\underline{Step 3:}
Let $K\supset k$ be a function field in one variable over $k$ and assume we are given  $f\in K^{\times}$ and $x_{1},\dots,x_{r}\in A(K)$. We need to show that
$$\sum_{v\;\rm{place\;of\;K/k}}\ord_{v}(f)(\sum_{j=0}^{r}(-1)^{r-j}\Tr_{k_{v}/k}(\sum_{1\leq\nu_{1}<\dots<\nu_{j}\leq r}[s_{v}(x_{\nu_{1}})+\dots+s_{v}(x_{\nu_{j}})]_{k(v)}))=0.$$
This will follow by the fact that for every place $v$ of $K$ over $k$, the map $s_{v}$ is a group homomorphism and by the following lemma.  \\
\\
\begin{lem}\label{ff}
For every $x\in A(K)$ it holds $\sum_{v}\ord_{v}(f)\Tr_{k_{v}/k}([s_{v}(x)]_{k_{v}})=0$, where the sum runs through all the places of $K$ over $k$.
\end{lem}
\begin{proof}
Let $C$ be the unique smooth projective curve that corresponds to the extension $K/k$. By the valuative criterion of properness, we obtain that the map $x:\Spec K\rightarrow A$ factors through the generic point inclusion $\eta:\Spec K\hookrightarrow C$ as follows:
\[ \xymatrix{
&\Spec K\ar[d]^{\eta} \ar[r]^{x}  & A\\
&C\ar[ru]^{\widetilde{x}}. \\
}
\] Here the map $\widetilde{x}:C\rightarrow A$ is given by $\widetilde{x}(v)=s_{v}(x)$. Since $\widetilde{x}$ is proper, it induces a push-forward map $$\widetilde{x}_{\star}:CH_{0}(C)\rightarrow CH_{0}(A).$$ Since $CH_{0}(C)=Pic(C)$ and $\dv(f)=0$ in $Pic(C)$, this yields $$\widetilde{x}_{\star}(\dv(f))=\sum_{v}\ord_{v}(f)\Tr_{k_{v}/k}([s_{v}(x)]_{k_{v}})=0.$$
\end{proof}
The last fact  completes the argument that the map $\Psi_{r}$ factors through $K_{r}(k;A)$. Finally, it is clear that if $\sigma$ is any permutation of the set $\{1,\dots,r\}$, it holds $\Psi_{r}(\{a_{1},\dots,a_{r}\})=\Psi_{r}(\{a_{\sigma(1)},\dots,a_{\sigma(r)}\})$. Therefore, we obtain a morphism $\displaystyle S_{r}(k;A)\stackrel{\Psi_{r}}{\longrightarrow} \frac{F^{r}}{F^{r+1}}$ as stated in the proposition.

\end{proof}
\medskip
\begin{rem}\label{r!} We observe that for every $r\geq 0$ the image of the map $\Psi_{r}$ is contained in $(G^{r}+F^{r+1})/F^{r+1}$. Furthermore, the composition
$$\Psi_{r}\circ\Phi_{r}:(G^{r}+F^{r+1})/F^{r+1}\rightarrow S_{r}(k;A)\rightarrow (G^{r}+F^{r+1})/F^{r+1}$$ is multiplication by $r!$.
\end{rem}
\medskip
\begin{cor}\label{S1} The canonical map  $A(k)\stackrel{\iota}{\longrightarrow} K_{1}(k;G)$ sending $a\in A(k)$ to the symbol $\{a\}_{k/k}$ is  an isomorphism.
\end{cor}
\begin{proof} It follows by lemma (\ref{ff}) that the inverse map
\begin{eqnarray*}&&K_{1}(k;G)\rightarrow A(k)\\
&&\{a\}_{k'/k}\rightarrow \Tr_{k'/k}(a)
\end{eqnarray*} is well defined.
\end{proof}
Our main Theorem now follows easily by the two previous propositions.
\medskip
\begin{thm}\label{Iso} Let $k$ be a field and $A$ an abelian variety over $k$. For the filtration $F^{r}CH_{0}(A)$ defined above, there are canonical isomorphisms of abelian groups:
$$\Phi_{r}:\mathbb{Z}[\frac{1}{r!}]\otimes\frac{F^{r}}{F^{r+1}}\stackrel{\simeq}{\longrightarrow}\mathbb{Z}[\frac{1}{r!}]
\otimes S_{r}(k;A), \;r\geq 1,$$ with $\displaystyle\Phi_{r}^{-1}=\frac{1}{r!}\Psi_{r}$. Moreover, the group $F^{2}CH_{0}(A)$ is precisely the Albanese kernel of $A$.
\end{thm}
\begin{proof} Definition (\ref{fil}) gives that $F^{r+1}=\ker\Phi_{r}|_{F^{r}}$. Thus, for every $r\geq 1$, we get an exact sequence
$$0\longrightarrow\frac{F^{r}}{F^{r+1}}\stackrel{\Phi_{r}}{\longrightarrow} S_{r}(k;A)\longrightarrow \frac{S_{r}(k;A)}{Im(\Phi_{r})}\longrightarrow 0.$$ Now notice that step 2 of proposition (\ref{psi}) yields an inclusion $Im(\Phi_{r})\supset r!S_{r}(k;A)$. Thus the group $S_{r}(k;A)/Im(\Phi_{r})$ is $r!$-torsion, which forces $(S_{r}(k;A)/Im(\Phi_{r}))\bigotimes\mathbb{Z}[\frac{1}{r!}]=0$. We conclude that after $\bigotimes\mathbb{Z}[\frac{1}{r!}]$, the map $\Phi_{r}$ becomes an isomorphism with inverse $\widetilde{\Psi}_{r}=\frac{1}{r!}\Psi_{r}$.\\
\label{F2}\\
Our next claim is that $F^{2}CH_{0}(A)=\ker\alb_{A}$. Using the isomorphism $A(k)\stackrel{\simeq}{\longrightarrow}K_{1}(k;A)$  (corollary \ref{S1}), the claim follows immediately from the commutative diagram
\[ \xymatrix{
&CH_{0}(A)/F^{2}\ar[r]^{\Phi_{0}\oplus\Phi_{1}} \ar[d]_{\deg\oplus\alb_{A}}  & \mathbb{Z}\oplus K_{1}(k;A)\ar[ld]^{\simeq}\\
& \mathbb{Z}\oplus A(k), \\
}
\] and the fact that $F^{2}$ is precisely the kernel of $\Phi_{0}\oplus\Phi_{1}$.

\end{proof}
\vspace{10pt}

\section{Properties of the Filtration}\label{4}
\subsection{The case $k=\overline{k}$}\label{closed}
\begin{prop}\label{algcl} If $A$ is an abelian variety over an algebraically closed field $k$, then for every $r\geq 1$ the groups $\displaystyle F^{r}CH_{0}(A)\otimes\mathbb{Z}[\frac{1}{(r-1)!}]$ and $\displaystyle G^{r}CH_{0}(A)\otimes\mathbb{Z}[\frac{1}{(r-1)!}]$ coincide.
\end{prop}
\begin{proof}
Notice that the  statement  holds trivially, if $r=1$. Let $r\geq 1$. Since the base field $k$ is algebraically closed, the group $S_{r}(k;A)$ is divisible, and we therefore have an equality $r!S_{r}(k;A)=S_{r}(k;A)$. Thus, for every $r\geq 1$, we obtain an isomorphism, $\Phi_{r}:F^{r}/F^{r+1}\stackrel{\simeq}{\longrightarrow} S_{r}(k;A)$. We will show by induction on $r$ that $\displaystyle F^{r}CH_{0}(A)\otimes\mathbb{Z}[\frac{1}{(r-1)!}]=G^{r}CH_{0}(A)\otimes\mathbb{Z}[\frac{1}{(r-1)!}]$.
Assume $\displaystyle F^{r}\otimes\mathbb{Z}[\frac{1}{(r-1)!}]=G^{r}\otimes\mathbb{Z}[\frac{1}{(r-1)!}]$ for some $r\geq 1$. Call $\overline{\Phi}_{r}:G^{r}/G^{r+1}\rightarrow S_{r}(k;A)$ the map induced by $\Phi_{r}$. Notice that the map
\begin{eqnarray*}&&\overline{\Psi}_{r}:S_{r}(k;A)\rightarrow G^{r}/G^{r+1}\\
&&\{a_{1},\dots,a_{r}\}\longrightarrow \sum_{j=0}^{r}(-1)^{r-j}\sum_{1\leq\nu_{1}<\dots<\nu_{j}\leq r}[a_{\nu_{1}}+\dots+a_{\nu_{j}}]
\end{eqnarray*} is well defined. The proof is essentially the same as the one of the well definedness of $\Psi_{r}$ (proposition \ref{psi}). Namely, steps 1 and 3 of the proof  apply directly in this setting, while step 2 is a tautology, since  there are no finite extensions of $k$, and hence no non-trivial Trace-Restriction maps.
We therefore obtain a commutative diagram as follows:
\[ \xymatrix{
&\displaystyle\frac{F^{r}}{F^{r+1}}\otimes\mathbb{Z}[\frac{1}{r!}]\ar[r]^{\simeq}  & \displaystyle S_{r}(k;A)\otimes\mathbb{Z}[\frac{1}{r!}]\\
&\displaystyle\frac{G^{r}}{G^{r+1}}\otimes\mathbb{Z}[\frac{1}{r!}]\ar[u]^{p}\ar[ru]^{\overline{\Phi}_{r}}, \\
}
\] where the map $p$ is the natural projection. The induction hypothesis clearly implies that $\displaystyle F^{r}\otimes\mathbb{Z}[\frac{1}{r!}]=G^{r}\otimes\mathbb{Z}[\frac{1}{r!}]$. Moreover, the composition
$$G^{r}/G^{r+1}\stackrel{\overline{\Phi}_{r}}{\longrightarrow}S_{r}(k;A)\stackrel{\overline{\Psi}_{r}}{\longrightarrow} G^{r}/G^{r+1}$$ is the multiplication by $r!$ (see remark \ref{r!}). In particular, after $\displaystyle \otimes\mathbb{Z}[\frac{1}{r!}]$, the map $\overline{\Phi}_{r}$ admits a section $\displaystyle\frac{1}{r!}\overline{\Psi}_{r}$.  We thus obtain an equality $\displaystyle F^{r+1}\otimes\mathbb{Z}[\frac{1}{r!}]=G^{r+1}\otimes\mathbb{Z}[\frac{1}{r!}]$.

\end{proof}
\begin{rem} We note that the filtration $G^{r}CH_{0}(A)$ has been studied before by S. Bloch, A. Beauville and others. We refer to \cite{Bl}, \cite{Beau1} and \cite{Beau2} for some results concerning this filtration.
\end{rem}
\begin{rem}\label{preserve} We now come back to the case of a non-algebraically closed field $k$. If $L\supset k$ is any field extension, the flat map $\pi_{L}:A_{L}\rightarrow A$ induces a pull back map $\res_{L/k}:CH_{0}(A)\rightarrow CH_{0}(A_{L})$, with $\res_{L/k}([a])=\sum_{\pi_{L}(\widetilde{a})=a}e_{\widetilde{a}}[\widetilde{a}]$, where $a$ is any closed point of $A$ and $e_{\widetilde{a}}$ is the length of the Artin local ring $A_{L}\times_{A}k(a)$ at $\widetilde{a}$. Notice that for $a\in A$ we have
$$\Tr_{L/k}(\res_{L/k}([a]))=(\sum_{\pi_{L}(\widetilde{a})=a}e_{\widetilde{a}})[a]=[L:k][a].$$
It is an immediate consequence of the definition of the restriction map between the $K$-groups (see functoriality in the subsection 2.2) that the following diagram commutes, for every $r\geq 0$.
\[ \xymatrix{
&CH_{0}(A)\ar[r]^{\res_{L/k}}\ar[d]^{\Phi_{r}}  & CH_{0}(A_{L})\ar[d]^{\Phi_{r}^{L}}\\
&S_{r}(k;A)\ar[r]^{\res_{L/k}} & S_{r}(L;A_{L}). \\
}
\] This in particular implies that the filtration $\{F^{r}\}_{r\geq 0}$ is preserved under restriction maps. For, if $x\in F^{r}CH_{0}(A)$, then by definition $\Phi_{r}(x)=0$. Thus,
$$\Phi_{r}^{L}(\res_{L/k}(x))=\res_{L/k}(\Phi_{r}(x))=0,$$ and therefore $\res_{L/k}(x)\in F^{r}CH_{0}(A_{L})$.\\ Moreover, if $L/k$ is finite, then filtration $\{G^{r}\}_{r\geq 0}$ is preserved under the Trace map $\Tr_{L/k}$. To see this, we observe that the generators of the group $G^{r}CH_{0}(A_{L})$ are of the form $$\sum_{j=0}^{r}(-1)^{r-j}\Tr_{L'/L}(\sum_{1\leq\nu_{1}<\dots<\nu_{j}\leq r}[a_{\nu_{1}}+\dots+a_{\nu_{j}}]_{L'}),$$ where $L'/L$ is a finite extension and $a_{i}\in A_{L}(L')$. Then
\begin{eqnarray*}&&\Tr_{L/k}(\sum_{j=0}^{r}(-1)^{r-j}\Tr_{L'/L}(\sum_{1\leq\nu_{1}<\dots<\nu_{j}\leq r}[a_{\nu_{1}}+\dots+a_{\nu_{j}}]_{L'}))=\\&&\sum_{j=0}^{r}(-1)^{r-j}\Tr_{L'/k}(\sum_{1\leq\nu_{1}<\dots<\nu_{j}\leq r}[\pi_{L}(a_{\nu_{1}})+\dots+\pi_{L}(a_{\nu_{j}})]_{L'}),
\end{eqnarray*} which is a generator of $G^{r}CH_{0}(A)$.
\end{rem}
\medskip
\begin{cor}\label{Q} If $A$ is an abelian variety over some field $k$, not necessarily algebraically closed, then the groups $F^{r}CH_{0}(A)\otimes\mathbb{Q}$ and $G^{r}CH_{0}(A)\otimes\mathbb{Q}$ coincide.
\end{cor}
\begin{proof} Let $x\in F^{r}CH_{0}(A)\otimes\mathbb{Q}$, for some $r\geq 0$. Then $x$ induces by restriction an element $\overline{x}=\res_{\overline{k}/k}(x)$ of $F^{r}CH_{0}(A_{\overline{k}})\otimes\mathbb{Q}$ (see remark \ref{preserve}). By proposition \ref{algcl}, we deduce that $\overline{x}\in G^{r}CH_{0}(A_{\overline{k}})\otimes\mathbb{Q}$ and we can therefore write it as $\displaystyle\overline{x}=\sum_{i=1}^{N}q_{i}x_{i}$, with $x_{i}\in G^{r}CH_{0}(A_{\overline{k}})$ and $q_{i}\in \mathbb{Q}$, for $i=1,\dots, N$. Let $L\supset k$ be a finite extension of $k$ such that all the $x_{i}$ are defined over $L$. Then we obtain:
$$\Tr_{L/k}(\res_{L/k}(x))=\sum_{i=1}^{N}q_{i}\Tr_{L/k}(x_{i})\in G^{r}CH_{0}(A)\otimes\mathbb{Q}.$$
The corollary then follows from the fact that $\Tr_{L/k}(\res_{L/k}(x))=[L:k]x$.

\end{proof}
\vspace{7pt}
\subsection{The finiteness of the Filtration}  Let $A$ be an abelian variety of dimension $d$ over some field $k$. In this section we elaborate the question if the filtration $F^{r}$ defined in section 3 stabilizes for large enough $r>0$. \\
We start by observing that the addition law on $A$, endows $CH_{0}(A)$ with a ring structure, by defining the Pontryagin product $[a]\star[b]=[a+b]$, for closed points $a,\;b$ of $A$. We can then easily see that the group $G^{r}CH_{0}(A)$ is the $r$-th power of the ideal $G^{1}$ of $(CH_{0}(A),\star)$ generated by elements of the form $\{[a]-[0],\;a\in A\}$.\\
\textbf{Fact:} The filtration $\{G^{r}\}_{r\geq 0}$ has the property  $G^{d+1}\otimes\mathbb{Q}=0$.\\
\\
S.Bloch in \cite{Bl} proves the above fact for the case of an algebraically closed base field, while A. Beauville in \cite{Beau2} gives a different proof for an abelian variety $A$ defined over $\mathbb{C}$. A few years later, C. Deninger and J.Murre in \cite{Den}, generalize Beauville's argument for an abelian variety over an arbitrary base field $k$, not necessarily algebraically closed.\\
\begin{cor} For every $r\geq d+1$, it holds $F^{r}CH_{0}(A)\otimes\mathbb{Q}=0$ and $S_{r}(k;A)\otimes\mathbb{Q}=0$, where $F^{r}$ is the filtration defined in section 3.
\end{cor}
\begin{proof} The first equality follows from corollary \ref{Q}, while the second from theorem \ref{Iso}.

\end{proof}
\begin{rem} We briefly recall the argument used by A. Beauville, and later by C. Deninger, J. Murre in their articles. A. Beauville uses the following Fourier Mukai transform:
\begin{eqnarray*}F:&&CH_{\bullet}(A)\otimes\mathbb{Q}\rightarrow CH_{\bullet}(\widehat{A})\otimes\mathbb{Q}\\
&&x\rightarrow \widehat{\pi}_{\star}(\exp(\mathcal{L})\cdot\pi^{\star}(x)),
\end{eqnarray*} where $\widehat{A}$ is the dual abelian variety of $A$, $\pi,\;\widehat{\pi}$ are the projections of $A\times\widehat{A}$ to $A$ and $\widehat{A}$ respectively, $\mathcal{L}$ is the Poincar\'{e} line bundle on $A\times\widehat{A}$ and the exponential $\exp(\mathcal{L})$ is defined as
$\displaystyle\exp(\mathcal{L})=\sum_{n=0}^{\infty}\frac{c_{1}(\mathcal{L})^{n}}{n!}$. Here we denote by $\cdot$ the intersection product in $CH_{\bullet}(A\times \widehat{A})$ and by $c_{1}(\mathcal{L})$  the image of $\mathcal{L}$ in $CH^{1}(A\times\widehat{A})$.\\
The map $F$ is induced by the Fourier Mukai isomorphism, $F_{D}:D(A)\rightarrow D(\widehat{A})$, between the derived categories of $A$ and $\widehat{A}$, defined by S. Mukai in \cite{Muk}, by first passing to the $K$-groups and then using the chern character isomorphism, $ch:K_{0}(A)\otimes\mathbb{Q}\stackrel{\simeq}{\longrightarrow} CH_{\bullet}(A)\otimes\mathbb{Q}$, where $CH_{\bullet}(A)$ is the Chow ring with operation the intersection product. The map $F$ has further the property of interchanging the intersection product of the ring $CH_{\bullet}(A)\otimes\mathbb{Q}$ with the Pontryagin product of $CH_{\bullet}(\widehat{A})\otimes\mathbb{Q}$. This property in turn implies that $G^{d+1}CH_{0}(A)\otimes\mathbb{Q}=0$, since $CH^{s}(\widehat{A})=0$ for $s>d$.\\
We believe that the above arguments will work after only $\otimes\Z[\frac{1}{(2d)!}]$. First,  notice that the Fourier Mukai transform $F$ can be considered as a map $\displaystyle F:CH_{\bullet}(A)\otimes\mathbb{Z}[\frac{1}{d!}]\rightarrow CH_{\bullet}(\widehat{A})\otimes\mathbb{Z}[\frac{1}{d!}]$, because the chern character isomorphism does hold after only $\displaystyle\otimes\mathbb{Z}[\frac{1}{d!}]$, (since $\displaystyle\frac{c_{1}(\mathcal{E})^{n}}{n!}=0$ for every $n>d$ and for every line bundle $\mathcal{E}$ on $A$). If after $\displaystyle\otimes\Z[\frac{1}{(2d)!}]$, the relative tangent bundle of the map $\widehat{\pi}:A\times\widehat{A}\rightarrow A$ is trivial as an element of $K_{0}(A\times\widehat{A})\otimes\Z[\frac{1}{(2d)!}]$, then by the Grothendieck Riemann-Roch theorem, the map $$F:CH_{\bullet}(A)\otimes\mathbb{Z}[\frac{1}{(2d)!}]\rightarrow CH_{\bullet}(\widehat{A})\otimes\mathbb{Z}[\frac{1}{(2d)!}]$$ will attain the above concrete description and will still
interchange the two products. This would imply that $\displaystyle G^{d+1}\otimes\mathbb{Z}[\frac{1}{(2d)!}]=0$ and further that  $\displaystyle F^{r}\otimes\mathbb{Z}[\frac{1}{(2d)!}]=F^{d+1}\otimes\mathbb{Z}[\frac{1}{(2d)!}]$, for every $r\geq d+1$.
\end{rem}
\vspace{7pt}
\subsection{An algorithm to compute generators of $F^{r}$} It is rather complicated to give a precise description of the generators of $F^{r}$, for $r\geq 3$, but things become much more concrete after $\displaystyle\otimes\mathbb{Z}[\frac{1}{r!}]$, because then the map $\Phi_{r-1}$ has a very concrete inverse, namely the map $\displaystyle\frac{1}{(r-1)!}\Psi_{r-1}$. In this section we will describe a recursive algorithm to compute generators of $\displaystyle F^{r}\otimes\mathbb{Z}[\frac{1}{(r-1)!}]$. As an application, we will give a complete set of generators of the Albanese kernel $F^{2}$ and of the group $\displaystyle F^{3}\otimes\mathbb{Z}[\frac{1}{2!}]$.\\
\\
\underline{Notation:}
If $k'\supset k$ is a finite extension and $a_{1},\dots,a_{r}\in A(k')$, we will denote by $w_{a_{1},\dots,a_{r}}$ the generator of $G^{r}$ corresponding to the $r-$tuple $(a_{1},\dots,a_{r})$, namely $$w_{a_{1},\dots,a_{r}}:=\sum_{j=0}^{r}(-1)^{r-j}\Tr_{k'/k}(\sum_{1\leq\nu_{1}<\dots<\nu_{j}\leq r}[a_{\nu_{1}}+\dots+a_{\nu_{j}}]_{k'}).$$
\begin{defn} Let $r\geq 1$. We consider the subgroup $R^{r+1}\subset F^{r}$ generated by the following two families of elements:
\begin{enumerate}\item For any finite extension $k'\supset k$ and $a_{1},\dots,a_{r+1}\in A(k')$, we require
\begin{eqnarray}\label{rel1}w_{a_{1},\dots,a_{r+1}}\in R^{r+1}.\end{eqnarray} (notice that this yields an inclusion $G^{r+1}\subset R^{r+1}$).
\item If $L\supset E\supset k$ is a tower of finite extensions, and we have elements $a_{i}\in A(L)$ for some $i\in\{1,\dots,r\}$, and $a_{j}\in A(E)$, for all $j\neq i$, then we require
    \begin{eqnarray}\label{Rel2}w_{a_{1},\dots,\Tr_{L/E}(a_{i}),\dots, a_{r}}-w_{\res_{L/E}(a_{1}),\dots,a_{i},\dots,\res_{L/E}(a_{r})}\in R^{r+1}.\end{eqnarray}
\end{enumerate}
\end{defn}
\begin{lem}\label{gen1} For every $r\geq 1$, $R^{r+1}$ is the smallest subgroup of $F^{r}$ that makes the homomorphism
\begin{eqnarray*}\Psi_{r}:\bigoplus _{k'/k}(A(k')\times A(k')\times\dots\times A(k'))&&\longrightarrow F^{r}/R^{r+1}\\
(a_{1},\dots,a_{r})_{k'/k}&&\longrightarrow w_{a_{1},\dots,a_{r}}
\end{eqnarray*} factor through $S_{r}(k;A)$. We therefore have an inclusion $R^{r+1}\subset F^{r+1}$, for every $r\geq 1$, and the composition $$S_{r}(k;A)\otimes\mathbb{Z}[\frac{1}{r!}]\stackrel{\frac{1}{r!}\Psi_{r}}{\longrightarrow} (F^{r}/R^{r+1})\otimes\mathbb{Z}[\frac{1}{r!}]\stackrel{\Phi_{r}}{\longrightarrow} S_{r}(k;A)\otimes\mathbb{Z}[\frac{1}{r!}]$$ is the identity map.
\end{lem}
\begin{proof} Notice that if $H\subset F^{r}$ is any subgroup, such that the map $$\Psi_{r}:\bigoplus _{k'/k}(A(k')\times A(k')\times\dots\times A(k'))\longrightarrow F^{r}/H$$ factors through $S_{r}(k;A)$, we definitely have $G^{r+1}\subset H$, since $\Psi_{r}$ needs to be multilinear. Further, elements of the form described in (\ref{Rel2}) above are necessarily in $H$, since $\{a_{1},\dots,\Tr_{L/E}(a_{i}),\dots, a_{r}\}_{E/k}=\{\res_{L/E}(a_{1}),\dots,a_{i},\dots,\res_{L/E}(a_{r})\}_{L/k}$ in $S_{r}(k;A)$. Therefore $R^{r+1}\subset H$. \\ We have no other restrictions, since if $K$ is a function field in one variable over $k$, then the relation $\sum_{v}\ord_{v}(f)\Tr_{k_{v}/k}(x)=0$ holds already in $CH_{0}(A)$, for $x\in A(K)$, and $f\in K^{\times}$ (see lemma \ref{ff}). The other statements follow directly from proposition \ref{psi}.

\end{proof}
We are now ready to describe our inductive argument.
\begin{prop}\label{gen2} For the Albanese kernel $F^{2}$ we have an equality $F^{2}=R^{2}$, and hence it can be generated by the following two families of elements:
\begin{enumerate}
\item For any finite extension $k'\supset k$, and points $a,\;b\in A(k')$, $$\Tr_{k'/k}([a+b]_{k'}-[a]_{k'}-[b]_{k'}+[0]_{k'})\in F^{2},$$
\item If $L\supset k$ is a finite extension, and $a\in A(L)$, then
$$\Tr_{L/k}([a]_{L}-[0]_{L})-([\Tr_{L/k}(a)]-[0])\in F^{2}.$$
\end{enumerate}
In general for $r\geq 2$ the group $\displaystyle F^{r}\otimes\mathbb{Z}[\frac{1}{(r-1)!}]$ can be generated by $\displaystyle R^{r}\otimes\mathbb{Z}[\frac{1}{(r-1)!}]$ and elements of the form $\displaystyle z-\frac{1}{(r-1)!}\Psi_{r-1}\circ\Phi_{r-1}(z)$, with $z\in F^{r-1}$.
\end{prop}
\begin{proof}
\underline{Step 1:} Compute generators of $F^{2}$: We already have an inclusion $R^{2}\subset F^{2}$ and an isomorphism $\Phi_{r}:F^{1}/F^{2}\simeq S_{1}(k;A)$. To show that $R^{2}\supset F^{2}$, it suffices therefore to prove commutativity of the following diagram.
\[ \xymatrix{
&(F^{1}/R^{2})\ar[d]^{1} \ar[r]^{\Phi_{1}}   & S_{1}(k;A)\ar[ld]_{\Psi_{1}}\\
&  (F^{1}/R^{2}).
}
\] (Notice that by lemma \ref{gen1} we have an equality $\Phi_{1}\circ\Psi_{1}=1$).
We need to verify this only for the generators of $F^{1}/R^{2}$, namely for $\Tr_{L/k}([a]-[0]_{L})$ with $a\in A(L)$. We have: $$\Psi_{1}\circ\Phi_{1}(\Tr_{L/k}([a]-[0]_{L}))=\Psi_{1}(\{a\}_{L/k})=\Tr_{L/k}([a]-[0]_{L}).$$
\underline{Step 2:}
Let $r\geq 3$. Consider the group $$B^{r}:=R^{r}\otimes\mathbb{Z}[\frac{1}{(r-1)!}]+
<z-\frac{1}{(r-1)!}\Psi_{r-1}\circ\Phi_{r-1}(z):z\in F^{r-1}>.$$ We want to show that
$\displaystyle F^{r}\otimes\mathbb{Z}[\frac{1}{(r-1)!}]=B^{r}$.\\
\underline{proof of $(\supset)$:} We already know $\displaystyle R^{r}\otimes\mathbb{Z}[\frac{1}{(r-1)!}]\subset F^{r}\otimes\mathbb{Z}[\frac{1}{(r-1)!}]$ (lemma \ref{gen1}). Moreover, if $z$ is any element of $F^{r-1}$, then \begin{eqnarray*}\Phi_{r-1}(z-\frac{1}{(r-1)!}\Psi_{r-1}\circ\Phi_{r-1}(z))&&=\Phi_{r-1}(z)-\Phi_{r-1}(\frac{1}{(r-1)!}
\Psi_{r-1}\circ\Phi_{r-1}(z))
\\&&=\Phi_{r-1}(z)-(\Phi_{r-1}\circ\frac{1}{(r-1)!}\Psi_{r-1})(\Phi_{r-1}(z))\\&&=\Phi_{r-1}(z)-\Phi_{r-1}(z)=0.
\end{eqnarray*} Thus, $\displaystyle z-\frac{1}{(r-1)!}\Psi_{r-1}\circ\Phi_{r-1}(z)\in\ker\Phi_{r-1}\cap F^{r-1}\otimes\mathbb{Z}[\frac{1}{(r-1)!}]$, which by definition is  $\displaystyle F^{r}\otimes\mathbb{Z}[\frac{1}{(r-1)!}]$.\\
\\
\underline{proof of $(\subset)$:}. Since the group $B^{r}$ contains $\displaystyle R^{r}\otimes\mathbb{Z}[\frac{1}{(r-1)!}]$, lemma (\ref{gen1}) implies that the map $$\Psi_{r-1}:S_{r-1}(k;A)\otimes\mathbb{Z}[\frac{1}{(r-1)!}]\rightarrow \frac{F^{r-1}\otimes\mathbb{Z}[\frac{1}{(r-1)!}]}{B^{r}}$$ is well defined and $\displaystyle\Phi_{r-1}\circ \frac{1}{(r-1)!}\Psi_{r-1}$ is the identity map. To complete the argument, it suffices to show that
$$\frac{1}{(r-1)!}\Psi_{r-1}\circ \Phi_{r-1}:
\frac{F^{r-1}\otimes\mathbb{Z}[\frac{1}{(r-1)!}]}{B^{r}}\rightarrow \frac{F^{r-1}\otimes\mathbb{Z}[\frac{1}{(r-1)!}]}{B^{r}}$$ is also the identity. This follows from the definition of $B^{r}$. Namely by definition, if  $z\in F^{r-1}$, then
$\displaystyle z-\frac{1}{(r-1)!}\Psi_{r-1}\circ\Phi_{r-1}(z)\in B^{r}$.

\end{proof}
\begin{rem}
Notice that the proposition (\ref{gen2}) describes a recursive algorithm to compute generators of the groups $\displaystyle F^{r}\otimes\mathbb{Z}[\frac{1}{(r-1)!}]$, for $r\geq 3$. Namely, having computed a complete set of generators of $\displaystyle F^{r}\otimes\mathbb{Z}[\frac{1}{(r-1)!}]$, the formula $\displaystyle F^{r+1}\otimes\mathbb{Z}[\frac{1}{r!}]=R^{r+1}\otimes\mathbb{Z}[\frac{1}{r!}]+
<z-\frac{1}{r!}\Psi_{r}\circ\Phi_{r}(z):z\in F^{r}>$ allows us to compute generators of $\displaystyle F^{r+1}\otimes\mathbb{Z}[\frac{1}{r!}]$. As an example, we compute below a set of generators of the group $\displaystyle F^{3}\otimes\mathbb{Z}[\frac{1}{2!}]$.
\end{rem}\medskip
\underline{Generators of $\displaystyle F^{3}\otimes\mathbb{Z}[\frac{1}{2!}]$:} According to proposition \ref{gen2}, we have the following families of generators:
\begin{enumerate}\item The ones that come from $\displaystyle R^{3}\otimes\mathbb{Z}[\frac{1}{2!}]$, namely:
\begin{enumerate}\item $\displaystyle\frac{1}{2^{m}}(\Tr_{k'/k}([a+b+c]_{k'}-[a+b]_{k'}-[a+c]_{k'}-[b+c]_{k'}+[a]_{k'}+[b]_{k'}+[c]_{k'}-[0]_{k'}))$, where $a,\;b,\;c\in A(k')$, and $m\geq 0$.
\item\label{rel2}
\begin{eqnarray*}&&\displaystyle\frac{1}{2^{m}}(\Tr_{E/k}([a+\Tr_{L/E}(b)]_{E}-[a]_{E}-
[\Tr_{L/E}(b)]_{E}+[0]_{E})-\\&&
    (\Tr_{L/k}([\res_{L/E}(a)+b]_{L}-[\res_{L/E}(a)]_{L}-
    [b]_{L}+[0]_{L}))),
\end{eqnarray*} where $L\supset E\supset k$ is a tower of finite extensions, $a\in A(E)$, $b\in A(L)$ and $m\geq 0$.
    \end{enumerate}
\item The ones that come from $\displaystyle z-\frac{1}{2}\Psi_{2}\circ\Phi_{2}(z)$ with $z\in F^{2}$.\\ Notice that if $z\in G^{2}$, then $\displaystyle z-\frac{1}{2}\Psi_{2}\circ\Phi_{2}(z)=0$, thus no new generator is obtained in this way. The only remaining generating family is of the form
$$\displaystyle\frac{1}{2^{m}}([\Tr_{L/k}(a)]-[0]-\Tr_{L/k}([a]_{L}-[0]_{L})-
\frac{1}{2}\Psi_{2}\Phi_{2}([\Tr_{L/k}(a)]-[0]-\Tr_{L/k}([a]_{L}-[0]_{L}))),$$ where $L\supset k$ is a finite  extension, $a\in A(L)$ and $m\geq 0$.
\end{enumerate}

\vspace{7pt}
\section{A cycle map to Galois cohomology}
In this section we recall the definition of the Somekawa map (\cite{Som}) which will in turn induce a cycle map to Galois cohomology. Let $n$ be an integer invertible in $k$ and $A$ an abelian variety  over $k$. We consider the connecting homomorphism $\delta:A/nA\rightarrow H^{1}(k,A[n])$ of the Kummer sequence of $A$,
$$0\longrightarrow A[n]\longrightarrow A\stackrel{\cdot n}{\longrightarrow}A\longrightarrow 0.$$ We will denote by $\cup$ the cup product pairing on $H^{\star}(k,A[n]^{\otimes \star})$. The Somekawa map is defined as follows:
\begin{eqnarray*}&&\frac{K_{r}(k;A)}{n}\stackrel{s_{n}}{\longrightarrow}H^{r}(k,A[n]^{\otimes r})\\
&&\{a_{1},\dots,a_{r}\}_{k'/k}\longrightarrow \Cor_{k'/k}(\delta(a_{1})\cup\dots\cup\delta(a_{r})),
\end{eqnarray*} where by $\Cor_{k'/k}$ we denote the correstriction map of Galois cohomology $H^{r}(k',A[n]^{\otimes r})\rightarrow H^{r}(k,A[n]^{\otimes r})$.
\begin{defn} Let $r\geq 1$ be a positive integer. We define the wedge product $\bigwedge^{r}A[n]$ as the cokernel of the map $0\rightarrow Sym^{r}(A[n])\rightarrow A[n]^{\otimes r}$, where $Sym^{r}(A[n])$ is the subgroup of $A[n]^{\otimes r}$ fixed by the action of $\Sigma_{r}$.
\end{defn}
\begin{prop} Let $A$ be an abelian variety over $k$ and let $n$ be an integer which is invertible in $k$. Then the Somekawa map induces
$$\frac{S_{r}(k;A)}{n}\stackrel{s_{n}}{\longrightarrow}H^{r}(k,\bigwedge^{r}A[n]).$$
\end{prop}
\begin{proof} The projection $A[n]^{\otimes r}\stackrel{p_{\wedge}}{\longrightarrow}\bigwedge^{r}A[n]$ induces a morphism
$H^{r}(k,A[n]^{\otimes r})\stackrel{p_{\wedge}}{\longrightarrow} H^{r}(k,\bigwedge^{r}A[n])$. Let $\{a_{1},...,a_{r}\}_{L/k}$ be any symbol in $K_{r}(k;A)$ and let $\sigma\in\Sigma_{r}$ be any permutation of the set $\{1,...,r\}$. We need to show that
$p_{\wedge}\circ s_{n}(\{a_{1},...,a_{r}\}_{L/k})=p_{\wedge}\circ s_{n}(\{a_{\sigma(1)},...,a_{\sigma(r)}\}_{L/k})$. Since any permutation $\sigma$ can be written as a product of transpositions of the form $\tau=(i,i+1)$, it suffices to show that for all $i\in\{1,...,r-1\}$, $$p_{\wedge}\circ s_{n}(\{a_{1},\dots,a_{i},a_{i+1},\dots,a_{r}\}_{L/k})=p_{\wedge}\circ s_{n}(\{\{a_{1},\dots,a_{i+1},a_{i},\dots,a_{r}\}_{L/k}).$$ We consider the map
\begin{eqnarray*}&&t:A[n]\otimes A[n]\otimes\dots\otimes A[n]\rightarrow A[n]\otimes A[n]\otimes\dots\otimes A[n]\\
&&a_{1}\otimes\dots \otimes a_{i}\otimes a_{i+1}\otimes \dots\otimes a_{r}\rightarrow a_{1}\otimes\dots\otimes  a_{i+1}\otimes a_{i}\otimes \dots\otimes a_{r}.
\end{eqnarray*} Then for the induced map $t_{\star}:H^{r}(k,A[n]^{\otimes^{r}})\rightarrow H^{r}(k,A[n]^{\otimes^{r}})$ it holds $$\Cor_{k'/k}(\delta(a_{1})\cup\dots\cup\delta(a_{i+1})\cup\delta(a_{i})\cup\dots\cup\delta(a_{r}))=
-t_{\star}(\Cor_{k'/k}(\delta(a_{1})\cup\dots\cup\delta(a_{i}))
\cup\delta(a_{i+1})\cup\dots\cup\delta(a_{r})).$$ (The last equality is a general fact about cup products in group cohomology. For a proof, we refer to \cite{Br}, p.111). Next notice that the following diagram is commutative:
\[ \xymatrix{
&H^{r}(k,\otimes^{r}A[n])\ar[r]^{t_{\star}} \ar[d]^{p_{\wedge}}   & H^{r}(k,\otimes^{r}A[n]) \ar[d]^{p_{\wedge}}\\
& H^{r}(k,\bigwedge^{r}A[n])\ar[r]^{-1} & H^{r}(k,\bigwedge^{r}A[n])\\
}
\]
To conclude, we have:
\begin{eqnarray*}p_{\wedge}(\Cor_{k'/k}(\delta(a_{1})\cup\dots\cup\delta(a_{i+1})\cup\delta(a_{i})
\cup\dots\cup\delta(a_{r})))=\\
p_{\wedge}(-t_{\star}(\Cor_{k'/k}(\delta(a_{1})\cup\dots\cup\delta(a_{i})\cup\delta(a_{i+1})\cup\dots\cup\delta(a_{r})))=\\
p_{\wedge}(\Cor_{k'/k}(\delta(a_{1})\cup\dots\cup\delta(a_{i})\cup\delta(a_{i+1})\cup\dots\cup\delta(a_{r}))).
\end{eqnarray*} The result now follows.
\end{proof}
\medskip
\begin{cor}\label{gal} For any integer $n$ invertible in $k$ and any $r>0$, the Somekawa map and the map $\Phi_{r}$ induce a cycle map to Galois cohomology: $$\frac{F^{r}CH_{0}(A)/F^{r+1}CH_{0}(A)}{n}\longrightarrow H^{r}(k,\bigwedge^{r}A[n]).$$
\end{cor}

\vspace{7pt}

\section{The $p$-adic Case} In all this section we assume that the base field $k$ is a finite extension of $\mathbb{Q}_{p}$, where $p$ is a prime number. Using results of W. Raskind and M. Spiess (\cite{Ras/Sp}), we obtain some divisibility results for our filtration. Furthermore, using the injectivity of the Galois symbol $K_{2}(k;A)/n\hookrightarrow H^{2}(k,A[n]^{\otimes 2})$  in the special case when $A$ has split multiplicative reduction, a result proved by T. Yamazaki  in (\cite{Yam}), we obtain a result for the Brauer group of $A$.
\subsection{Divisibility Results}
\begin{rem} We recall that if  $A$ is an abelian variety over the $p$-adic field $k$ and $\mathcal{A}$ is the N\'{e}ron model of $A$, then we say that $A$ has semi-abelian reduction if the connected  component, $\mathcal{A}_{s}^{0}$ , of the special fiber $\mathcal{A}_{s}$  of $\mathcal{A}$ containing the neutral element fits into a short exact sequence
$$0\rightarrow T\rightarrow \mathcal{A}_{s}^{0}\rightarrow B\rightarrow 0,$$ where $T$ is a torus and $B$ an abelian variety over the residue field $\kappa$ of $k$.
Further, we say that $A$ has split semi-ordinary reduction, if it has semi-abelian reduction with $T$ a split torus and $B$ an ordinary abelian variety.
\end{rem}
Raskind and Spiess obtained the following important result.
\begin{thm} (\cite{Ras/Sp} theorem 4.5) Let $A_{1},...,A_{n}$ be abelian varieties over $k$ with split semi-ordinary reduction. Then for $n\geq 2$, the group $K(k;A_{1},....,A_{n})$ is the direct sum of a finite group $F$ and a divisible group $D$. For $n\geq 3$, the group $K(k;A_{1},....,A_{n})$ is in fact divisible (remark 4.4.5).
\end{thm}
Thus, in our set up, if we assume that the abelian variety $A$ has split semi-ordinary reduction, then theorem \ref{Iso} has the following corollary:
\begin{cor}\label{div} Let $A$  be an abelian variety over a $p$-adic field $k$ having split semi-ordinary reduction. Then for the filtration defined above, it holds:
\begin{enumerate}\item For $r\geq 3$, the groups $F^{r}/F^{r+1}$ are divisible.
\item The group $\displaystyle F^{2}/F^{3}\otimes\mathbb{Z}[\frac{1}{2}]$ is a direct sum of a divisible group and a finite group.
\end{enumerate}
\end{cor}
\begin{proof} Everything follows directly from theorem \ref{Iso}, once we notice that for $r\geq 3$ the divisibility of $S_{r}(k;A)$ yields an equality $S_{r}(k;A)=r!S_{r}(k;A)$. Thus, the injective map $F^{r}/F^{r+1}\stackrel{\Phi_{r}}{\hookrightarrow}S_{r}(k;A)$ is also surjective.
\end{proof}
\medskip

\vspace{7pt}
\subsection{The Brauer group}
In this section we compute the kernel of the map $$CH_{0}(A)\otimes\Z[\frac{1}{2}]\rightarrow Br(A)^{\star}\otimes\Z[\frac{1}{2}]$$ induced by the Brauer-Manin pairing, in the special case when the abelian variety $A$ has split multiplicative reduction. First, we review some definitions.
\begin{enumerate}
\item Let $X$ be a smooth, projective, geometrically connected variety over the $p$-adic field $k$. By the Brauer group of $X$ we will always mean the group $H^{2}(X_{\et},\mathbb{G}_{m})$ and we will denote it by $Br(X)$. There is a well defined pairing of abelian groups $$<,>_{X}:CH_{0}(X)\times Br(X)\rightarrow\mathbb{Q/\Z}$$ defined as follows. If $\alpha\in Br(X)$ is an element of the Brauer group and $x\in X$ a closed point of $X$, then the closed immersion $\iota_{x}:\Spec(k(x))\rightarrow X$ induces the pullback $\iota_{x}^{\star}:Br(X)\rightarrow Br(k(x))$. We define $$<x,\alpha>_{X}=\Cor_{k(x)/k}(\iota_{x}^{\star}(\alpha))\in Br(k)\simeq\mathbb{Q}/\Z,$$ where $\Cor_{k(x)/k}:Br(k(x))\to Br(k)$ is the Correstriction map and the isomorphism $Br(k)\simeq\mathbb{Q}/\Z$ is via the invariant map of local class field theory. To show that this definition factors through rational equivalence, we reduce to the case of curves, where the well definedness of the pairing follows by a result of S. Lichtenbaum, \cite{Licht}. \\
\item We say that an abelian variety $A$ of dimension $d$ over $k$ has split multiplicative reduction, if the connected component $\mathcal{A}_{s}^{0}$, containing the neutral element of the special fiber $\mathcal{A}_{s}$ of the N\'{e}ron model $\mathcal{A}$ of $A$ is a split torus. In this case, the theory of degeneration of abelian varieties (\cite{Fal}, Chapter III, Proposition 8.1) yields that there exists a split torus $T\simeq \mathbb{G}_{m}^{\oplus d}$ over $k$ and a finitely generated free abelian group $L\subset T(k)$ of rank $d$, such that for any finite extension $k'/k$, there is an isomorphism $A(k')\simeq T(k')/L$.
\end{enumerate}
\medskip
We will first prove a lemma which holds for any smooth, projective, geometrically connected variety $X$ over $k$, which connects the pairing $<,>_{X}$ and the usual cycle map to \'{e}tale cohomology,
$$\rho_{X,n}:CH_{0}(X)/n\rightarrow H^{2d}(X,\mathbb{Z}/n(d)).$$
\underline{Notation:} We denote by $\mathbb{Z}/n(d)$ the abelian sheaf $\mu_{n}^{\otimes d}$ on $X_{\et}$.
\begin{lem}\label{ya} Let $X$ be a smooth, projective, geometrically connected variety over $k$. There is a commutative diagram
\[ \xymatrix{
& CH_{0}(X)/n \ar[r]_{\rho_{X,n}}\ar[d]_{<,>_{X}}   &   H^{2d}(X_{\et},\mathbb{Z}/n(d))\ar[d]\\
& (Br(X)[n])^{\star}\ar[r] & (H^{2}(X,\mu_{n}))^{\star}. \\
}
\]
\end{lem}
\begin{proof} First we observe that Tate and Poincar\'{e} duality induce a non-degenerate pairing of finite abelian groups $$H^{2}(X,\mu_{n})\times H^{2d}(X,\mathbb{Z}/n(d))\rightarrow \mathbb{Z}/n.$$ (see \cite{Saito} for a proof of this statement, due to S. Saito). Notice that this pairing induces the right vertical map of the diagram stated in the lemma, $H^{2d}(X,\mathbb{Z}/n(d))\stackrel{\simeq}{\longrightarrow}H^{2}(X,\mu_{n})^{\star}$, which is therefore an isomorphism. \\
Let now $x$ be a closed point of $X$. We obtain a commutative diagram
\[ \xymatrix{
& H^{0}(x,\mathbb{Z}/n) \ar[r]^{G_{x}}\ar[d]_{\simeq}   &   H^{2d}(X_{\et},\mathbb{Z}/n(d))\ar[d]\\
& (H^{2}(x,\mu_{n}))^{\star}\ar[r]^{\iota_{x}^{\star}} & (H^{2}(X,\mu_{n}))^{\star}, \\
}
\] where the left vertical map is the isomorphism induced by Tate duality, $H^{2}(X,\mu_{n})\stackrel{\iota_{x}^{\star}}{\longrightarrow} H^{2}(x,\mu_{n})$ is the pullback map and $G_{x}$ is the Gysin map. Recall that the cycle map  $\rho_{X,n}:CH_{0}(X)/n\rightarrow H^{2d}(X,\mathbb{Z}/n(d))$ is defined by $\rho_{X,n}([x])=G_{x}(1)$. The result now follows from the commutative diagram
\[ \xymatrix{
& H^{2}(X,\mu_{n}) \ar[r]^{\iota_{x}^{\star}}\ar[d]   &   H^{2}(x,\mu_{n})\ar[d]_{\simeq}\\
& Br(X)[n]\ar[r]^{\iota_{x}^{\star}} & Br(k(x))[n]. \\
}
\] Here the two vertical maps arise from the Kummer sequence on $X_{\et}$ and $x_{\et}$ respectively. Notice that the commutativity of the last diagram follows from the functoriality properties of \'{e}tale cohomology (universality of the functor $H^{\star}(X_{\et},\_)$). \\
We can thus conclude that the map $H^{0}(x,\Z/n)\rightarrow H^{2d}(X,\mathbb{Z}/n(d))\rightarrow (H^{2}(X,\mu_{n}))^{\star}$ factors through $H^{0}(x,\Z/n)\rightarrow (Br(X)[n])^{\star}$  and the lemma follows.

\end{proof}
\begin{cor}\label{kernel} The kernel of the map $CH_{0}(X)/n\stackrel{<,>_{X}}{\longrightarrow}(Br(X)[n])^{\star}$ coincides with the kernel of the cycle map $\rho_{X,n}:CH_{0}(X)/n\rightarrow H^{2d}(X,\mathbb{Z}/n(d))$.
\end{cor}
\begin{proof} This follows immediately from the commutative diagram of lemma \ref{ya}, as soon as we notice that the right vertical map is an isomorphism, and the bottom horizontal map is injective. The injectivity of $(Br(X)/n)^{\star}\rightarrow(H^{2}(X,\mu_{n}))^{\star}$ follows by applying the exact functor $\rm{Hom}(\_,\mathbb{Q}/\Z)$ to the s.e.s.
$$0\rightarrow\rm{Pic}(X)/n\rightarrow H^{2}(X,\mu_{n})\rightarrow Br(X)[n]\rightarrow 0,$$ arising from the Kummer sequence for $X$.

\end{proof}
\medskip
\textbf{\underline{The Hochschild-Serre spectral sequence:}}
We now go back to the case of  an abelian  variety $A$ of dimension $d$ over the $p$-adic field $k$ . We consider the Hochschild-Serre spectral sequence, $$E_{2}^{pq}=H^{p}(k,H^{q}(A_{\overline{k}},\mathcal{F}))\Rightarrow H^{p+q}(A,\mathcal{F}),$$ where $\mathcal{F}$ is any abelian sheaf on $A_{\et}$. For any $q\geq0$, the spectral sequence  gives a descending filtration $$H^{q}(A_{\et},\mathcal{F})=H^{q}_{0}\supset H^{q}_{1}\supset\dots H^{q}_{q-1}\supset H^{q}_{q}\supset 0,$$ with quotients
$H^{q}_{i}/H^{q}_{i+1}\simeq E^{i,q-i}_{\infty}.$ First we observe that $H^{q}_{i}=0,$ for $i\geq 3$. For, the $p$-adic field $k$ has cohomological dimension 2, which forces $E_{2}^{i,q-i}$ to be zero for $i\geq 3$. We will use this filtration for the groups $H^{2d}(A,\Z/n(d))$ and $Br(A)=H^{2}(A,\mathbb{G}_{m})$.\\
\begin{lem}\label{m*} After $\displaystyle\otimes\Z[\frac{1}{2}]$, the spectral sequence $$H_{2}^{pq}=H^{p}(k,H^{q}(A_{\overline{k}},\mathbb{Z}/n(d)))\Rightarrow H^{p+q}(A,\mathbb{Z}/n(d))$$ degenerates at level 2.
\end{lem}
\begin{proof} We need to show that all the differentials $d_{2}^{pq}$ become zero after $\displaystyle\otimes\Z[\frac{1}{2}]$. The statement is clear when $p\geq 1$ or $p<0$ or $q<1$ even before $\displaystyle\otimes\Z[\frac{1}{2}]$. We will show that for $q\geq 1$, the map $$d_{2}^{0,q}:H^{0}(k,H^{q}(\overline{A},\mathbb{Z}/n(d)))\rightarrow H^{2}(k,H^{q-1}(\overline{A},\mathbb{Z}/n(d)))$$ has the property $2d_{2}^{0,q}=0$. Let $m\in\Z$ be an integer. We consider the multiplication by $m$ map $A\stackrel{m}{\longrightarrow} A$ on $A$. The map $m$ induces a pull back map on cohomology,
$$H^{p}(k,H^{q}(A_{\overline{k}},\mathbb{Z}/n(d)))\stackrel{m^{\star}}{\longrightarrow}H^{p}(k,H^{q}(A_{\overline{k}},
\mathbb{Z}/n(d))),$$ for every $p,q$. Moreover, since $m$ is a morphism of schemes, the pullback $m^{\star}$ is compatible with the differentials, i.e. the following diagram commutes, for every $q\geq 1$.
\[ \xymatrix{
& H^{0}(k,H^{q}(A_{\overline{k}},\mathbb{Z}/n(d))) \ar[r]^{m^{\star}}\ar[d]^{d^{0,q}}   &   H^{0}(k,H^{q}(A_{\overline{k}},\mathbb{Z}/n(d)))\ar[d]^{d^{0,q}}\\
& H^{2}(k,H^{q-1}(A_{\overline{k}},\mathbb{Z}/n(d)))\ar[r]^{m^{\star}} & H^{2}(k,H^{q-1}(A_{\overline{k}},\mathbb{Z}/n(d))). \\
}
\]
The action of $m^{\star}$ on $H^{0}(k,H^{q}(A_{\overline{k}},\mathbb{Z}/n(d)))$ is multiplication by $m^{q}$. For, the action is induced by the action of $m^{\star}$ on $H^{1}(A_{\overline{k}},\mathbb{Z}/n)=\Hom(A[n],\Z/n)$, which is multiplication by $m$.
Let $\alpha\in H^{0}(k,H^{q}(\overline{A},\mathbb{Z}/n(d)))$. Taking $m=-1$ and using the fact that  $d^{0,q}$ is a group homomorphism, we get
$$d^{0,q}((-1)^{\star}(\alpha))=d^{0,q}((-1)^{q}\alpha)=(-1)^{q}d^{0,q}(\alpha).$$ On the other hand, using the
commutativity of the diagram above, we obtain
$$d^{0,q}((-1)^{\star}(\alpha))=(-1)^{\star}d^{0,q}(\alpha)=(-1)^{q-1}d^{0,q}(\alpha).$$ We conclude that
$d^{0,q}(\alpha)=-d^{0,q}(\alpha)$ and hence $2d^{0,q}(\alpha)=0$.

\end{proof}
\begin{cor}\label{ss} The filtration $H_{0}^{2d}\supset H_{1}^{2d}\supset H_{2}^{2d}\supset 0$ of the group $H^{2d}(A,\Z/n(d))$ induced by the Hochschild-Serre spectral sequence has successive quotients: $H_{0}^{2d}/H_{1}^{2d}\simeq H^{0}(k,H^{2d}(A,\Z/n(d)))$, $H_{1}^{2d}/H_{2}^{2d}\simeq H^{1}(k,H^{2d-1}(A,\Z/n(d)))$ and $\displaystyle H_{2}^{2d}\otimes\Z[\frac{1}{2}]\simeq  H^{2}(k,H^{2d-2}(A,\Z/n(d)))\otimes\Z[\frac{1}{2}]$.
\end{cor}
\begin{proof} The third equality follows directly from lemma \ref{m*}. We claim that for $p=0,1$, $E_{\infty}^{p,2d-p}=E_{2}^{p,2d-p}$ before $\displaystyle\otimes\Z[\frac{1}{2}]$.
For $p=1$ the statement follows immediately from the observation that both the differentials $d_{2}^{1,2d-1}$ and $d_{2}^{-1,2d}$ are zero.\\
For $p=0$, we first observe that  $E_{\infty}^{0,2d}=E_{3}^{0,2d}=\ker d_{2}^{0,2d}$. For, the map $d_{3}^{0,2d}:E_{3}^{0,2d}\rightarrow E_{3}^{3,2d-3}$ is the zero map, since $E_{3}^{3,2d-3}=0$. Thus, we have an inclusion
$$H_{0}^{2d}/H_{1}^{2d}=\ker d_{2}^{0,2d}\stackrel{j}{\hookrightarrow} E_{2}^{0,2d}=H^{0}(k,H^{2}(A_{\overline{k}},\Z/n(d))).$$
Since $A$ is projective, Poincar\'{e} duality yields an isomorphism
$$H^{2d}(A_{\overline{k}},\Z/n(d))\simeq \rm{Hom}(H^{0}(A_{\overline{k}},\Z/n),\mathbb{Z}/n)\simeq\Z/n.$$
We therefore obtain the following commutative diagram:
\[ \xymatrix{
& & H_{0}^{2d} \ar[d]_{p} & CH_{0}(A)/n\ar[l]^{\rho_{A,n}} \ar[d]^{\deg} \\
&0\ar[r]   & H_{0}^{2d}/H_{1}^{2d} \ar[d]\ar[r]^{j} & \Z/n \ar[d] \\
& &0 & 0
}
\]
Notice that since $A$ is an abelian variety, there exists a $k$-rational point, and hence the degree map is surjective. Since $\deg=j\circ p\circ\rho_{A,n}$, we conclude that the map $j$ is surjective.

\end{proof}
\medskip
\begin{prop}\label{rho} Let $A$ be an abelian variety over $k$ and $n\geq 1$ a positive integer. The cycle map $$\rho_{A,n}:CH_{0}(A)/n\rightarrow H^{2d}(A,\mathbb{Z}/n(d)),$$ when restricted to $F^{3}/n$ is the zero map. Moreover, if $A$ has split multiplicative reduction, then after $\displaystyle\otimes\Z[\frac{1}{2}]$, the kernel of the cycle map is precisely the group $\displaystyle((F^{3}+nCH_{0}(A))/n)\otimes\Z[\frac{1}{2}]$.
\end{prop}
\begin{proof} Consider the filtration $H_{0}^{2d}\supset H^{2d}_{1}\supset H^{2d}_{2}\supset 0$ of the group $H^{2d}(A,\Z/n(d))$. Then from the commutative diagram
\[ \xymatrix{
&0 \ar[r] & F^{1}/n \ar[r] \ar[r] & CH_{0}(A)/n\ar[d]^{\rho_{A,n}} \ar[r]   &   \Z/n\ar[d]^{\simeq}\ar[r] &0\\
&0\ar[r]   & H_{1}^{2d} \ar[r] & H^{2d}(A,\Z/n(d))\ar[r] & H_{0}^{2d}/H_{1}^{2d}\ar[r]& 0 \\
}
\] we obtain that $F^{1}/n$ is mapped to $H_{1}^{2d}$ via the cycle map and the kernel of $\rho_{A,n}$ is contained in $F^{1}/n$. Notice that the right vertical map is an isomorphism by lemma \ref{ss}.
Next we consider the commutative diagram
\[ \xymatrix{
& & F^{2}/n \ar[r] \ar[r] & F^{1}/n\ar[d]^{\rho_{A,n}} \ar[r]   &   (F^{1}/F^{2})/n\ar@{^{(}->}[d]\ar[r] &0\\
&0\ar[r]   & H_{2}^{2d} \ar[r] & H_{1}^{2d}\ar[r] & H_{1}^{2d}/H_{2}^{2d}\ar[r]& 0 \\
}
\] from where we obtain that $F^{2}/n$ maps to $H_{2}^{2d}$ via the cycle map and the kernel of $\rho_{A,n}$ is contained in the image of the map $F^{2}/n\to F^{1}/n$. Notice that in this case, the right vertical map is injective. For,  by Poincar\'{e} duality and the \'{e}tale cohomology of abelian varieties over an algebraically closed field, we obtain isomorphisms
\begin{eqnarray*}H^{2d-1}(A_{\overline{k}},\Z/n(d))\simeq H^{1}(A_{\overline{k}},\Z/n)(-1)\simeq
\rm{Hom}(A[n],\Z/n)(-1)\simeq A[n].
\end{eqnarray*} and therefore the map $(F^{1}/F^{2})/n\to H_{1}^{2d}/H_{2}^{2d}$ coincides with the map $A(k)/n\hookrightarrow H^{1}(k,A[n])$ arising from the Kummer sequence for $A$, $0\rightarrow A[n]\rightarrow A\stackrel{n}{\longrightarrow} A\rightarrow 0$. \\
Next we turn our attention to the map $\rho_{A,n}:F^{2}/n\rightarrow H_{2}^{2d}$. Again, by Poincar\'{e} duality we obtain
\begin{eqnarray*}&&H^{2d-2}(A_{\overline{k}},\Z/n(d))\simeq \rm{Hom}(H^{2}(A_{\overline{k}},\Z/n),\Z/n)\simeq\\&&
\rm{Hom}(\wedge^{2}(\rm{Hom}(A[n],\Z/n),\Z/n)\simeq \wedge^{2}A[n].
\end{eqnarray*} Thus, the cycle map induces $F^{2}/n\stackrel{\rho_{A,n}}{\longrightarrow}H^{2}(k,\wedge^{2}A[n])$.\\ Using now the map $s_{n}:(F^{2}/F^{3})/n\rightarrow H^{2}(k,\wedge^{2}A[n])$ obtained in corollary  (\ref{gal}), we deduce that $\rho_{A,n}:F^{2}/n\rightarrow H_{2}^{2d}$ factors through $(F^{2}/F^{3})/n$ and therefore the group $F^{3}/n$, being the kernel of $F^{2}/n\to(F^{2}/F^{3})/n$, is contained in the kernel of the map $\rho_{A,n}$. This concludes the proof of the first statement of the proposition.\\
Assume now that $A$ has split multiplicative reduction. We will prove that the  map
$$(F^{2}/F^{3})/n\otimes\Z[\frac{1}{2}]\rightarrow H^{2}(k,\wedge^{2}A[n])\otimes\Z[\frac{1}{2}]$$ is injective.
By theorem (\ref{Iso}), it suffices to prove that the Somekawa map
$$s_{n}:S_{2}(k;A)/n\otimes\Z[\frac{1}{2}]\rightarrow H^{2}(k,\wedge^{2}A[n])\otimes\Z[\frac{1}{2}]$$ is injective. T.Yamazaki, in \cite{Yam}, proved that in the split multiplicative reduction case, the map $$s_{n}:K_{2}(k;A)/n\rightarrow H^{2}(k,A[n]\otimes A[n])$$ is injective. Consider the following commutative diagram
\[ \xymatrix{
& K_{2}(k;A)/n \ar@{^{(}->}[r]^{s_{n}} \ar[d]   &   H^{2}(k,A[n]\otimes A[n]) \ar[d]\\
&  S_{2}(k;A)/n \ar[r]^{s_{n}} & H^{2}(k,\wedge^{2}A[n]).
}
\]  Notice that after $\displaystyle\otimes\Z[\frac{1}{2}]$ both vertical maps have sections. Namely, the maps
\begin{eqnarray*}&&S_{2}(k;A)/n\otimes\Z[\frac{1}{2}]\stackrel{i}{\rightarrow} K_{2}(k;A)/n\otimes\Z[\frac{1}{2}]\\
&&\{a,b\}_{k'/k}\rightarrow \frac{\{a,b\}_{k'/k}+\{b,a\}_{k'/k}}{2}
\end{eqnarray*}
and \begin{eqnarray*}&&H^{2}(k,\wedge^{2}A[n])\otimes\Z[\frac{1}{2}]\stackrel{j}{\rightarrow} H^{2}(k,
A[n]\otimes A[n])\otimes\Z[\frac{1}{2}],
\end{eqnarray*} induced by the map
\begin{eqnarray*}&&\wedge^{2}A[n]\otimes\Z[\frac{1}{2}]\rightarrow A[n]\otimes A[n]\otimes\Z[\frac{1}{2}]\\
&& x\wedge y\rightarrow \frac{x\otimes y-y\otimes x}{2}.
\end{eqnarray*} The injectivity of the map $$s_{n}:S_{2}(k;A)/n\otimes\Z[\frac{1}{2}]\rightarrow H^{2}(k,\wedge^{2}A[n])\otimes\Z[\frac{1}{2}]$$ hence follows from the commutative diagram
\[ \xymatrix{
&\displaystyle K_{2}(k;A)/n\otimes\Z[\frac{1}{2}] \ar@{^{(}->}[r]^{s_{n}}    & \displaystyle   H^{2}(k,A[n]\otimes A[n])\otimes\Z[\frac{1}{2}] \\
& \displaystyle  S_{2}(k;A)/n\otimes\Z[\frac{1}{2}] \ar@{^{(}->}[u]^{i} \ar[r]^{s_{n}} &\displaystyle  H^{2}(k,\wedge^{2}A[n])\otimes\Z[\frac{1}{2}]\ar@{^{(}->}[u]^{j}.
}
\]

\end{proof}

\begin{thm} Let $A$ be an abelian variety over $k$. The subgroup $F^{3}$ is contained in the kernel of the map
$$j:CH_{0}(A)\rightarrow Br(A)^{\star}.$$ If moreover $A$ has split multiplicative reduction, then the kernel of the map $$CH_{0}(A)\otimes\Z[\frac{1}{2}]\stackrel{j\otimes\mathbb{Z}[\frac{1}{2}]}{\longrightarrow} Br(A)^{\star}\otimes\Z[\frac{1}{2}]$$ is the subgroup $D$ of $\displaystyle F^{2}\otimes\Z[\frac{1}{2}]$, which contains $\displaystyle F^{3}\otimes\Z[\frac{1}{2}]$ and is such that $\displaystyle D/(F^{3}\otimes\Z[\frac{1}{2}])$ is the maximal divisible subgroup of $\displaystyle F^{2}/F^{3}\otimes\Z[\frac{1}{2}]$.
\end{thm}
\begin{proof} Assume to contradiction that $F^{3}\varsubsetneq\ker j$ and let $w\in F^{3}$ be such that $j(w)\neq 0$. This means that there exists some element $\alpha\in Br(A)$ such that $<w,\alpha>\neq 0$. Notice that the group $Br(A)$ is torsion, because it is a subgroup of $Br(K)$, where $K$ is the function field of $A$ (for a proof of the last statement see \cite{Grot}, II, 1.10). Let $m$ be the order of $\alpha$. Then $j(w)$ gives a nonzero morphism $Br(A)[m]\rightarrow\mathbb{Q}/\Z$.
The map $F^{3}\to Br(A)[m]^{\star}$ factors through $F^{3}/m$ and by proposition (\ref{rho}) , we get a commutative diagram  \[ \xymatrix{
& F^{3}/m \ar[d] \ar[r]^{0}   &   H^{2d}(X_{\et},\mathbb{Z}/n(d))\ar[d]\\
& Br(X)[m]^{\star}\ar[r] & (H^{2}(X,\mu_{n}))^{\star}. \\
}
\]
Since the bottom map is injective, we conclude that the map $F^{3}/m\to Br(A)[m]^{\star}$ is zero, which is the desired contradiction. \\ Next, proposition \ref{div} gives us an isomorphism $$F^{2}/F^{3}\otimes\Z[\frac{1}{2}]\simeq D_{0}\oplus F_{0},$$ where $F_{0}$ is a finite group and $D_{0}$ is divisible. Let $D$ be the subgroup of $\displaystyle F^{2}\otimes\Z[\frac{1}{2}]$ such that $\displaystyle D/(F^{3}\otimes\Z[\frac{1}{2}])\simeq D_{0}$. It is clear that $\displaystyle D\subset\ker(CH_{0}(A)\otimes\Z[\frac{1}{2}]\rightarrow Br(A)^{\star}\otimes\Z[\frac{1}{2}])$, since $Br(A)$ is a torsion group. \\
Assume now that $A$ has split multiplicative reduction. We will show that $D$ is in fact equal to $\ker(j\otimes\frac{1}{2})$. First, we consider the filtration $H_{0}^{2}\supset H_{1}^{2}\supset H_{2}^{2}\supset 0$ of $Br(A)$ arising from the Hochschild-Serre spectral sequence, $H^{p}(k,H^{q}(\overline{A},\mathbb{G}_{m}))\Rightarrow H^{p+q}(A,\mathbb{G}_{m})$.\\
We can easily see that $E_{\infty}^{1,1}=E_{2}^{1,1}$. For, both the differentials  $d_{2}^{1,1}$ and $d_{2}^{-1,2}$ are zero. This yields an isomorphism
$H_{1}^{2}/H_{2}^{2}\simeq H^{1}(k,H^{1}(A_{\overline{k}},\mathbb{G}_{m}))$. Next we observe that
$E_{\infty}^{2,0}=E_{3}^{2,0}=E_{2}^{2,0}/\img(E_{2}^{0,1}\rightarrow E_{2}^{2,0})$. For, both the differentials $d_{3}^{2,0}$ and $d_{3}^{-1,2}$ are zero. In particular, we have a surjection  $E_{2}^{2,0}\rightarrow H_{2}^{2}\rightarrow 0$. Dualizing, we obtain an inclusion $0\rightarrow (H_{2}^{2})^{\star}\rightarrow(E_{2}^{2,0})^{\star}.$ Since $A$ is proper, we have an isomorphism $$E_{2}^{2,0}\simeq H^{2}(k,H^{2}(A_{\overline{k}},\mathbb{G}_{m}))\simeq Br(k)\simeq \mathbb{Q}/\Z.$$
We have a commutative diagram as follows:
\[ \xymatrix{
&0 \ar[r] & F^{1}\ar[r]  & CH_{0}(A)\ar[d]^{j} \ar[r]^{\deg}   &   \Z\ar[d]\ar[r] &0\\
&0\ar[r]   & (H_{0}^{2}/H_{2}^{2})^{\star} \ar[r] & (H_{0}^{2})^{\star}\ar[r] & (H_{2}^{2})^{\star}\ar[r]& 0 \\
}
\]  We claim that the right vertical map is an inclusion. To see this, we observe that the composition
$$\Z\rightarrow(H_{2}^{2})^{\star}\hookrightarrow(E_{2}^{2,0})^{\star}$$  coincides with the inclusion $\Z\hookrightarrow\widehat{\Z}=(Br(k))^{\star}$. (We note here that T.Yamazaki is using this exact same argument for the injectivity in the proof of his proposition 3.1 in \cite{Yam}). We conclude that $\ker j\subset F^{1}$. Moreover, under this map, $F^{1}$ is sent to the subgroup $(H_{0}^{2}/H_{2}^{2})^{\star}$. Next we consider the commutative diagram
\[ \xymatrix{
&0 \ar[r] & F^{2}\ar[r]  & F^{1}\ar[d]^{j} \ar[r]^{\alb_{A}}   &   A(k)\ar[d]\ar[r] &0\\
&0\ar[r]   & (H_{0}^{2}/H_{1}^{2})^{\star} \ar[r] & (H_{0}^{2}/H_{2}^{2})^{\star}\ar[r] & (H_{1}^{2}/H_{2}^{2})^{\star}\ar[r]& 0. \\
}
\] We claim that the right vertical map is again an injection. For, we have an isomorphism $H_{1}^{2}/H_{2}^{2}\simeq E_{2}^{1,1}\simeq H^{1}(k,\rm{Pic}(A_{\overline{k}}))$. Moreover, Tate duality yields an isomorphism
$$A(k)^{\star}\simeq H^{1}(k,\widehat{A})=H^{1}(k,\rm{Pic}^{0}(A_{\overline{k}})).$$ (see \cite{Mil2}). Applying $\rm{Hom}(\_,\mathbb{Q}/\Z)$, we obtain an injection $$A(k)\hookrightarrow (H^{1}(k,\rm{Pic}^{0}(A_{\overline{k}})))^{\star}.$$ We conclude, that since the composition $$A(k)\rightarrow (H^{1}(k,\rm{Pic}(A_{\overline{k}})))^{\star}\rightarrow(H^{1}(k,\rm{Pic}^{0}(A_{\overline{k}})))^{\star}$$ is injective, the first map needs to be injective as well. This yields an inclusion $\ker j\subset F^{2}$, and therefore $\displaystyle\ker ( j\otimes\Z[\frac{1}{2}])\subset F^{2}\otimes\Z[\frac{1}{2}]$. Next, notice that the map $\displaystyle j\otimes\Z[\frac{1}{2}]$ induces $$(F^{2}\otimes\Z[\frac{1}{2}])/D\stackrel{j\otimes\Z[\frac{1}{2}]}{\longrightarrow} Br(A)^{\star}\otimes\Z[\frac{1}{2}].$$ To see that this last map is injective, let $n$ be the order of $\displaystyle(F^{2}\otimes\Z[\frac{1}{2}])/D=F_{0}$. Since $D$ is divisible, we have an equality $$\frac{(F^{2}/F^{3})\otimes\Z[\frac{1}{2}]}{n}=F_{0}.$$ Since the kernel of the map $$j\otimes\Z[\frac{1}{2}]:\frac{(F^{2}/F^{3})\otimes\Z[\frac{1}{2}]}{n}\rightarrow Br(A)[n]^{\star}\otimes\Z[\frac{1}{2}]$$  coincides with the kernel of the cycle map $$\rho_{A,n}\otimes\Z[\frac{1}{2}]:\frac{(F^{2}/F^{3})\otimes\Z[\frac{1}{2}]}{n}\rightarrow H^{2d}(A,\Z/n(d))\Z[\frac{1}{2}]$$
(corollary (\ref{kernel})), the result follows by the second part of proposition (\ref{rho}).

\end{proof}
\medskip
\begin{rem} We conjecture that if the abelian variety $A$ has semi-ordinary reduction, the group $F^{3}$ is divisible. This would mean that in the special case of split multiplicative reduction, the cycle map $\displaystyle\rho_{A,n}:CH_{0}(A)/n\otimes\Z[\frac{1}{2}]\rightarrow H^{2d}(A,\Z/n(d))\otimes\Z[\frac{1}{2}]$ is injective and the kernel of $$j:CH_{0}(A)\otimes\Z[\frac{1}{2}]\rightarrow Br(A)^{\star}\otimes\Z[\frac{1}{2}]$$ is the maximal divisible subgroup of $\displaystyle CH_{0}(A)\otimes\Z[\frac{1}{2}]$.
\end{rem}

\vspace{20pt}
\textbf{Acknowledgement:} I would like to express my great gratitude to my advisor, Pr. Kazuya Kato, who very kindly adviced me and provided me with multiple hints in order to obtain the above paper. He also pointed out many mistakes and gave me a very helpful feedback during the writing procedure. Moreover, I am really grateful for the discussion I had with Pr. Alexander Beilinson regarding the Fourier Mukai transform. Finally, I would like to thank Pr. Spencer Bloch and Pr. Madhav Nori for their interest in my work as well as my referee whose comments and suggestions improved greatly this paper.
{99}

\medskip
\medskip
\medskip
Department of Mathematics,\\ University of Chicago, 5734 University Ave.\\ Chicago, Illinois 60637\\
email: valiagaz@math.uchicago.edu

\end{document}